\documentclass{amsart}

\usepackage{ae,aecompl}
\usepackage{chngcntr}
\usepackage{graphicx}
\usepackage[all,cmtip]{xy}
\usepackage{amsmath, amscd}
\usepackage{amsthm}
\usepackage{amssymb}
\usepackage{amsfonts}
\usepackage{qsymbols}
\usepackage{latexsym}
\usepackage{mathrsfs}
\usepackage{cite}
\usepackage{color}
\usepackage{url}
\usepackage{enumerate}
\usepackage{hyperref}
\usepackage{verbatim}

\allowdisplaybreaks

\newcommand {\abs}[1]{\lvert#1\rvert}

\newcommand {\Be}{{B}}

\newcommand {\C}{{\mathbb C}}

\newcommand {\Ce}{\mathrm{C}}

\newcommand {\ud}{\mathrm{d}}
\newcommand {\ue}{\mathrm{e}}
\newcommand {\eps}{\varepsilon}

\newcommand {\Ell}{L}
\newcommand {\Ellp}{L^{p}}
\newcommand {\Ellpprime}{L^{p'}}
\newcommand {\Ellq}{L^{q}}

\newcommand {\Ellone}{L^{1}}
\newcommand {\Elltwo}{L^{2}}
\newcommand {\Ellinfty}{L^{\infty}}
\newcommand {\Ellinftyc}{L_{\mathrm{c}}^{\infty}}
\newcommand {\F}{{\mathcal{F}}}

\newcommand {\ui}{\mathrm{i}}

\newcommand {\ind}{{\mathbf{1}}}

\newcommand{\lb}{\langle}
\newcommand{\rb}{\rangle}
\newcommand {\La}{{\mathcal{L}}}
\newcommand {\calL}{{\mathcal{L}}}

\newcommand {\N}{{{\mathbb N}}}

\newcommand {\norm}[1]{\left\|#1\right\|}
\newcommand{\one}{{{\bf 1}}}
\newcommand {\ph}{{\varphi}}
\newcommand {\Pa}{{\mathcal{P}}}
\newcommand {\R}{{\mathbb R}}
\newcommand {\Rd}{{\mathbb{R}^{d}}}

\newcommand {\supp}{{\mathrm{supp}}}

\newcommand {\Sw}{\mathcal{S}}

\newcommand {\T}{{\mathbb T}}

\newcommand {\Z}{{{\mathbb Z}}}

\newcommand {\vanish}[1]{\relax}

\newcommand{\wh}{\widehat}

\renewcommand{\restriction}{\mathord{\upharpoonright}}

\newtheorem{theorem}{Theorem}[section]
\newtheorem{lemma}[theorem]{Lemma}
\newtheorem{proposition}[theorem]{Proposition}
\newtheorem{corollary}[theorem]{Corollary}

\theoremstyle{definition}

\newtheorem{remark}[theorem]{Remark}

\numberwithin{equation}{section}

\title[Fourier multiplier theorems on Besov spaces]{Fourier multiplier theorems on Besov spaces under type and cotype conditions}

\author{Jan Rozendaal}
\address{Institute of Mathematics Polish Academy of Sciences\\
ul.~\'{S}niadeckich 8\\
00-656 Warsaw\\
Poland}
\email{janrozendaalmath@gmail.com}

\author{Mark Veraar}
\address{Delft Institute of Applied
Mathematics\\
Delft University of Technology\\
P.O.~Box 5031\\
2628 CD Delft\\
The Netherlands}
\email{M.C.Veraar@tudelft.nl}

\subjclass[2010]{Primary: 42B15; Secondary: 42B35, 46B20, 46E40, 47B38}

\keywords{Operator-valued Fourier multipliers, Besov spaces, type and cotype, Fourier type, extrapolation}

\thanks{The second author is supported by the VIDI subsidy 639.032.427 of the Netherlands Organisation for Scientific Research (NWO)}

\begin{document}
\begin{abstract}
In this paper we consider Fourier multiplier operators between vector-valued Besov spaces with different integrability exponents $p$ and $q$, which depend on the type $p$ and cotype $q$ of the underlying Banach spaces. In a previous paper we considered $L^p$-$L^q$-multiplier theorems. In the current paper we show that in the Besov scale one can obtain results with optimal integrability exponents. Moreover, we derive a sharp result in the $L^p$-$L^q$-setting as well.

We consider operator-valued multipliers without smoothness assumptions. The results are based on a Fourier multiplier theorem for functions with compact Fourier support. If the multiplier has smoothness properties then the boundedness of the multiplier operator extrapolates to other values of $p$ and $q$ for which $\frac1p - \frac1q$ remains constant.
\end{abstract} \maketitle

\section{Introduction}

In this paper we consider Fourier multiplier operators $T_{m}(f)  = \F^{-1} (m \F f)$ on vector-valued Besov spaces. Here $\F$ denotes the Fourier transform and $m$ is an operator-valued function on $\Rd$. In \cite{Rozendaal-Veraar16Fourier} we considered such operators on vector-valued $L^p$-spaces. The advantage of the Besov scale is that boundedness of the Fourier multiplier operator can be obtained with optimal integrability exponents $p$ and $q$, where $p$ is the type of $X$ and $q$ the cotype of $Y$.

In the case $p=q$, Fourier multiplier operators on vector-valued Besov spaces have been considered in \cite{Girardi-Weis03, Hytonen04} and in \cite{Arendt-Bu-period} in the periodic setting. In both papers it is shown that under Fourier type assumptions on $X$, one can obtain boundedness results under less restrictive smoothness conditions on the multipliers than in the $L^{p}$-scale. Moreover, it was shown by Amann in \cite{Amann97} and Weis in \cite{Weis97} that the UMD condition on the underlying space, which is required for multiplier theorems in the $L^{p}$-scale, can be avoided in the Besov scale. Similar results on Triebel--Lizorkin spaces have been obtained in \cite{Bu-Kim05a,Bu-Kim05b}. In \cite{Shahmurov10} some of the results of \cite{Girardi-Weis03} have been extended to the setting where $p\neq q$.

We aim to prove Fourier multiplier results on Besov spaces without any smoothness conditions on the multiplier $m$. Our main result is as follows (for type and cotype see Section \ref{sec:type and cotype assumptions}, for the definition of Besov spaces and the dyadic annuli $I_{k}$ see Section \ref{Besov spaces}):

\begin{theorem}\label{main result Besov multipliers typeintro}
Let $X$ be a Banach space with type $p\in[1,2]$ and $Y$ a Banach space with cotype $q\in[2,\infty]$, and let $r\in[1,\infty]$ be such that $\tfrac{1}{r}=\tfrac{1}{p}-\tfrac{1}{q}$. Let $m:\Rd\to\La(X,Y)$ be an $X$-strongly measurable map such that $\left(2^{k\sigma}\gamma(\{m(\xi)\!\mid\! \xi\in I_{k}\})\right)_{k\in\N_{0}}\in\ell^{u}$ for some $\sigma\in\R$ and $u\in[1,\infty]$. Then there exists a constant $C\geq 0$ independent of $m$ such that $T_{m}$ extends to a bounded linear map $\widetilde{T_{m}}:\Be^{s}_{p,v}(\Rd;X) \to \Be^{s+\sigma-d/r}_{q,w}(\Rd;Y)$
with
\begin{align*}
\|\widetilde{T_{m}}\|_{\La(\Be^{s}_{p,v}(\Rd;X),\Be^{s+\sigma-d/r}_{q,w}(\Rd;Y))}\!\leq C\norm{\Big(2^{k\sigma}\gamma(\{m(\xi)\mid \xi\in I_{k}\})\Big)_{k}}_{\ell^{u}}
\end{align*}
for all $s\in\R$ and all $v,w\in[1,\infty]$ with $\frac{1}{w}\leq \frac{1}{u}+\frac{1}{v}$.
\end{theorem}

For a proof of this result see Theorem \ref{main result Besov multipliers type}. If $m$ is scalar-valued then the $\gamma$-bound reduces to a uniform bound. A similar result is derived under Fourier type assumptions and in that case the $\gamma$-bound can also be replaced by a uniform bound. A version of Theorem \ref{main result Besov multipliers typeintro} in the $L^p$-$L^q$-scale was obtained in \cite{Rozendaal-Veraar16Fourier}, where it is assumed that $X$ has type $p_{0}>p$ and $Y$ cotype $q_{0}<q$. The proof of Theorem \ref{main result Besov multipliers typeintro} is based on an $L^p$-$L^q$ Fourier multiplier result for functions with compact Fourier support. As a corollary of our results on Besov spaces we also obtain a multiplier theorem in the $L^{p}$-$L^{q}$-scale.

Under smoothness conditions on $m$ (which depend on the Fourier type of $X$ and $Y$) the boundedness result extends to all values of $1<p\leq q<\infty$ such that $\frac{1}{p}-\frac1q = \frac1r$. The latter statement was given in \cite{Rozendaal-Veraar16Fourier} without proof. Here we present the proof which is an extension of the extrapolation results of the classical paper of H\"ormander \cite{Hormander60} to the case $p\leq q$. Part of our extrapolation result is new even in the scalar case.

Fourier multiplier theorems on vector-valued Besov spaces have found applications to boundary value problems, maximal regularity, the stability theory for $C_{0}$-semigroups and functional calculus theory (see \cite{Amann97, Arendt-Bu02, Weis97, Haase-Rozendaal16}). The results in this paper have already been applied in \cite{Rozendaal15} and in the forthcoming paper \cite{Rozendaal-Veraar16Stability}.

The paper is organized as follows. In Section \ref{preliminaries} we discuss preliminaries for the rest of the paper. In Section \ref{sec:Fourierm} we introduce operator-valued Fourier multipliers on vector-valued function spaces, and we consider some properties which are specific to multipliers on Besov spaces. In Section \ref{sec:type and cotype assumptions} we prove our main multiplier theorems on Besov spaces and derive a corollary in the $L^p$-scale. Then in Section \ref{extrapolation} we prove our extrapolation results, first with conditions on the kernel of the Fourier multiplier operator and then with conditions on the symbol of the operator.

\subsection{Notation and terminology}

The natural numbers are $\N:=\left\{1,2,3,\ldots\right\}$, and $\N_{0}:=\N\cup\left\{0\right\}$.

Nonzero Banach spaces over the complex numbers are denoted by $X$ and $Y$, and the space of bounded linear operators from $X$ to $Y$ is $\La(X,Y)$. We set $\La(X):=\La(X,X)$, and we write $\mathrm{I}_{X}$ for the identity operator on $X$.

For $p\in[1,\infty]$ and a measure space $(\Omega,\mu)$, we let $\Ellp(\Omega;X)$ be the Bochner space of equivalence classes of strongly measurable $X$-valued functions on $\Omega$ which are $p$-integrable. When, for a map $f:\Omega\to X$, we write $\|f\|_{\Ell^{p}(\Omega;X)}<\infty$ then it is implicitly assumed that $f$ is strongly measurable. We denote by $p'$ the H\"{o}lder conjugate of $p$, which is defined by $1=\frac{1}{p}+\frac{1}{p'}$. We let $\ell^{p}$ be the space of $p$-summable sequences $(x_{k})_{k\in\N_{0}}\subseteq\C$ over $\N_{0}$, while $\ell^{p}(\Z)$ is the space of $p$-summable sequences $(x_{k})_{k\in\Z}\subseteq\C$ over $\Z$.

A function $m:\Omega\to\La(X,Y)$ is said to be \emph{$X$-strongly measurable} if $\omega\mapsto m(\omega)x$ is strongly measurable as a map from $\Omega$ to $Y$ for all $x\in X$. Throughout we will identify a scalar function $m:\Rd\to\C$ with the associated operator-valued function $\widetilde{m}:\Rd\to\La(X)$ given by $\widetilde{m}(\xi):=m(\xi)\mathrm{I}_{X}$ for $\xi\in\Rd$.

For $d\in\N$ the class of $X$-valued Schwartz functions is $\Sw(\Rd;X)$, and $\Sw'(\Rd;X)$ is the space of $X$-valued tempered distributions. We let $\Sw(\Rd):=\Sw(\Rd;\C)$ and we denote by $\lb \cdot,\cdot\rb:\Sw'(\Rd;X)\times \Sw(\Rd)\to X$ the $X$-valued duality between $\Sw'(\Rd;X)$ and $\Sw(\Rd)$. The Fourier transform of a $\Phi\in\Sw'(\Rd;X)$ is denoted by $\F \Phi$ or $\widehat{\Phi}$, and its inverse Fourier transform by $\F^{-1}\Phi$ or $\check{\Phi}$. The Fourier transform is normalized as follows:
\begin{align*}
\widehat{f}(\xi)=\F f(\xi):=\int_{\Rd}\ue^{-2\pi\ui \xi \cdot t}f(t)\,\ud t
\end{align*}
for $f\in\Ell^{1}(\Rd;X)$ and $\xi\in\Rd$. By $\supp(\Phi)\subseteq\Rd$ we denote the distributional support of $\Phi\in\Sw'(\Rd;X)$. For $\Omega\subseteq\Rd$ we define
\begin{align}\label{analytic Schwartz functions}
\Sw_{\Omega}(\Rd;X):=\{f\in\Sw(\Rd;X)\mid \supp(\widehat{f}\,)\subseteq\Omega\}\subseteq \Sw(\Rd;X)
\end{align}
and, for $p\in[1,\infty]$,
\begin{align}\label{analytic Lp-functions}
\Ell^{p}_{\Omega}(\Rd;X):=\{f\in\Ellp(\Rd;X)\mid \supp(\widehat{f}\,)\subseteq\Omega\}\subseteq \Ellp(\Rd;X).
\end{align}

A complex standard Gaussian random variable on a probability space $(\Omega,\mathbb{P})$ is a random variable $\gamma$ of the form $\gamma=\frac{\gamma_{r}+\ui \gamma_{i}}{\sqrt{2}}$, where $\gamma_{r},\gamma_{i}:\Omega\to\R$ are independent real standard Gaussians on $\Omega$. A \emph{Gaussian sequence} is a sequence $(\gamma_{k})_{k}$ (finite or infinite) of independent complex standard Gaussian random variables on some probability space.

\section{Preliminaries on function spaces}\label{preliminaries}

In this section we present some of the background on function space theory which will be used throughout the paper.

\subsection{Besov spaces}\label{Besov spaces}

We first define vector-valued Besov spaces. For more details on these spaces see e.g.~\cite{Amann97,Bergh-Lofstrom, Triebel10}.

Throughout this section, fix $d\in\N$. Let $\psi\in\Sw(\R)$ be such that
\begin{align}\label{Schwartz function defining Besov spaces}
\supp(\widehat{\psi})\subseteq[\tfrac{1}{2},2],\quad \widehat{\psi}\geq 0\quad{\text{and}} \quad\sum_{k=-\infty}^{\infty}\widehat{\psi}(2^{-k}\xi)=1\quad(\xi\in(0,\infty)).
\end{align}
For $k\in\N$, define
\begin{align}\label{dyadic annuli}
I_{k}:=\{\xi\in\Rd\mid 2^{k-1}\leq \abs{\xi}\leq 2^{k+1}\}\quad \text{and}\quad I_{0}:=\{\xi\in\Rd\mid \abs{\xi}\leq 2\}.
\end{align}
Moreover, let $(\ph_{k})_{k\in\N_{0}}\subseteq\Sw(\Rd)$ be such that
\begin{align}\label{Littlewood-Paley functions}
\widehat{\ph_{k}}(\xi)=\widehat{\psi}(2^{-k}\abs{\xi})\quad\text{for }k\in\N\quad \text{and}\quad \widehat{\ph_{0}}(\xi)=1-\sum_{k=1}^{\infty}\widehat{\ph_{k}}(\xi)
\end{align}
for all $\xi\in\Rd$. For notational simplicity we let $\ph_{k}:= 0$ for $k<0$. Then $\sum_{k=0}^{\infty}\widehat{\ph_{k}}=1$ for all $\xi\in\Rd$, and for all $k\in\N_{0}$ it holds that $\supp(\widehat{\ph_{k}})\subseteq I_{k}$, $\widehat{\ph_{k}}(\xi)=0$ if $\xi\in I_{n}$ for $n\notin \{k-1,k,k+1\}$ and $\widehat{\ph_{k-1}}(\xi)+\widehat{\ph_{k}}(\xi)+\widehat{\ph_{k+1}}(\xi)=1$ if $\xi\in\supp(\widehat{\ph_{k}})$. Throughout this article we keep the function $\psi$ from \eqref{Schwartz function defining Besov spaces} and the sequence $(\ph_{k})_{k\in\N_{0}}\subseteq\Sw(\Rd)$ from \eqref{Littlewood-Paley functions} fixed.

Let $X$ be a Banach space and let $s\in\R$ and $p,v\in[1,\infty]$. The \emph{inhomogeneous Besov space} $\Be^{s}_{p,v}(\Rd;X)$ is the space of all $f\in\Sw'(\Rd;X)$ such that $\ph_{k}\ast f\in\Ellp(\Rd;X)$ for all $k\in\N_{0}$ and
\begin{align*}
\norm{f}_{\Be^{s}_{p,v}(\Rd;X)}:=\Big\|\Big(2^{ks}\big\|\ph_{k}\ast f\big\|_{\Ellp(\Rd;X)}\Big)_{k\in\N_{0}}\Big\|_{\ell^{v}}<\infty,
\end{align*}
endowed with the norm $\norm{\cdot}_{\Be^{s}_{p,v}(\Rd;X)}$. Then $\Be^{s}_{p,v}(\R;X)$ is a Banach space and the continuous inclusions
\begin{align*}
\Sw(\Rd;X)\subseteq\Be^{s}_{p,v}(\Rd;X)\subseteq\Sw'(\Rd;X)
\end{align*}
hold. Here the second embedding has dense range, as does the first embedding if $p,v\in[1,\infty)$. A different choice of $\psi$ satisfying \eqref{Schwartz function defining Besov spaces} would yield an equivalent norm on $\Be^{s}_{p,v}(\Rd;X)$. Generally $s$ is called the \emph{smoothness index} of $\Be^{s}_{p,v}(\Rd;X)$.

For $p\in[1,\infty]$, $s,t\in\R$ with $t<s$ and $v,w\in[1,\infty]$ with $v\leq w$, the following embeddings hold:
\begin{align}\label{embedding between Besov spaces}
\Be^{s}_{p,v}(\Rd;X)\subseteq \Be^{s}_{p,w}(\Rd;X)\subseteq\Be^{t}_{p,1}(\Rd;X).
\end{align}
Here the first embedding is a contraction and the norm of the second embedding is independent of $X$.

For later use we note, as is straightforward to check, that there exist constants $C_{1},C_{2}\in(0,\infty)$ such that, for each Banach space $X$ and all $p,v\in[1,\infty]$, $s\in\R$, $n\in\N$ and $f\in\Ellp(\Rd;X)$ with $\supp(\widehat{f}\,)\subseteq I_{n}$,
\begin{align}\label{Lp-norm on dyadic parts}
C_{1}2^{(n-1)\abs{s}}\|f\|_{\Ellp(\Rd;X)}\leq \|f\|_{\Be^{s}_{p,v}(\Rd;X)}\leq C_{2}2^{(n+1)\abs{s}}\|f\|_{\Ellp(\Rd;X)}.
\end{align}

\medskip

We will also consider homogeneous Besov spaces. To define these we first introduce vector-valued homogeneous distributions. Let
\begin{align*}
\dot{\Sw}(\Rd;X):=\{f\in\Sw(\Rd;X)\mid \mathrm{D}^{\alpha}\!\widehat{f}(0)=0\text{ for all } \alpha\in\N_{0}^{d}\}.
\end{align*}
Endow $\dot{\Sw}(\Rd;X)$ with the subspace topology of $\Sw(\Rd;X)$ and let $\dot{\Sw}(\Rd):=\dot{\Sw}(\Rd;\C)$. Let $\dot{\Sw}'(\Rd;X)$ be the space of continuous linear mappings from $\dot{\Sw}(\Rd)$ to $X$. Each $f\in \Sw'(\Rd;X)$ induces an $f\restriction_{\dot{\Sw}(\Rd)}\in\dot{\Sw}'(\Rd;X)$ by restriction, and for $f,g\in \Sw'(\Rd;X)$ one has $f\restriction_{\dot{\Sw}(\Rd)}=g\restriction_{\dot{\Sw}(\Rd)}$ if and only if $\supp(\widehat{f}-\widehat{g})\subseteq\{0\}$. Conversely, the following lemma shows that each $f\in \dot{\Sw}'(\Rd;X)$ extends to an element of $\Sw'(\Rd;X)$.

\begin{lemma}\label{lem:extenddistri}
Let $X$ be a Banach space and let $u\in \dot{\Sw}'(\Rd;X)$. Then there exists a $\widetilde{u}\in \Sw'(\R^d;X)$ such that $\widetilde{u}\restriction_{\dot{\Sw}(\Rd)}=u$.
\end{lemma}
In the scalar case the statement of the lemma is a straightforward consequence of the Hahn-Banach theorem. Unfortunately, in the vector-valued setting one cannot argue in this way.
\begin{proof}
Let $k\in \N_0$ be such that
\begin{align}\label{eq:ucontext}
\|u f\|_X \leq C \!\sum_{|\alpha|, |\beta|\leq k} \|x^{\alpha} D^{\beta} f\|_{\infty}
\end{align}
for all $f\in \dot{\Sw}(\R^d)$. By an approximation argument $u$ can be extended to
\[
\Sw_k(\R^d) := \{f\in \Sw(\Rd)\mid D^{\alpha} \hat{f}(0) = 0 \text{ for all }  |\alpha|\leq k\}.
\]
Indeed, to see this by \eqref{eq:ucontext} it suffices to show that $\dot{\Sw}(\R^d)$ is dense in $\Sw_k(\R^d)$ with respect to the norm
\begin{equation}\label{eq:Sk-norm}
\sum_{|\alpha|, |\beta|\leq k} \|x^{\alpha} D^{\beta} (\cdot)\|_{\infty}.
\end{equation}
Let $\varphi\in C^\infty(\R^d)$ be such that $\varphi(\xi) = 1$ if $|\xi|\geq 2$, $\varphi(\xi) = 0$ if $|\xi|\leq 1$, and $0\leq \varphi\leq 1$.
For $n\in\N$ and $\xi\in\Rd$ let $\varphi_n(\xi) := \varphi(n\xi)$. For $f\in \Sw_k(\R^d)$ let $f_n := \F^{-1}(\varphi_n) *f$.  By Taylor's theorem there exists a constant $C\geq 0$ such that
\[
\|D^{\gamma} \wh f(\xi)\| \leq C |\xi|^{k+1 - |\gamma|}
\]
for every $|\gamma|\leq k$ and $\xi\in\Rd$ with $\abs{\xi}\leq 1$.
Using this one readily checks that $f_n\to f$ in the norm \eqref{eq:Sk-norm}.

Finally, we extend $u$ from $\Sw_k(\R^d)$ to $\Sw(\R^d)$. In order to do so fix $g_{\beta}\in \Sw(\R^d)$ such that $D^{\alpha} \wh g_{\beta}(0) = 1$ if $\alpha = \beta$ and zero if $\alpha \neq \beta$. Now let $f\in\Sw(\Rd)$ and let $L_k f\in \Sw(\R^d)$ be given by
\[
L_kf (\xi) := \sum_{|\beta|\leq k} g_\beta(\xi) D^{\beta}\hat{f}(0).
\]
Then $f- L_k f\in \Sw_k(\R^d)$ and we can define $v f := u(f - L_kf)\in X$. Then $vf = uf$ if $f\in \Sw_k(\R^d)$, since $L_k$ vanishes on $\Sw_k(\R^d)$. Moreover, by \eqref{eq:ucontext},
\begin{align*}
\|vf\|_X &\leq C \sum_{|\alpha|, |\beta|\leq k} \|x^{\alpha} D^{\beta} (f-L_kf)\|_{\infty}
\\ & \leq C \sum_{|\alpha|, |\beta|\leq k} \|x^{\alpha} D^{\beta} f\|_{\infty} + C \sum_{|\alpha|, |\beta|\leq k} \|x^{\alpha} D^{\beta} L_kf\|_{\infty}
\\ & \leq \tilde{C} \sum_{|\alpha|, |\beta|\leq k} \|x^{\alpha} D^{\beta} f\|_{\infty}
\end{align*}
for a constant $\tilde{C}\geq 0$. Hence $v\in \Sw'(\R^d;X)$ and the proof is concluded.
\end{proof}

It follows from Lemma \ref{lem:extenddistri} and the statements preceding it that $\dot{\Sw}'(\Rd;X)=\Sw'(\R^d;X)/\Pa(\Rd;X)$, where $\Pa(\Rd;X):=\{f\in\Sw'(\Rd;X)\mid \supp(\widehat{f})\subseteq\{0\}\}$. Moreover, $\Pa(\Rd;X)=\Pa(\Rd)\otimes X$ for $\Pa(\Rd)$ the polynomials on $\Rd$, as can be shown in the same way as \cite[Proposition 2.4.1]{Grafakos08}. If $F(\Rd;X)\subseteq\Sw'(\Rd;X)$ is a linear subspace such that, for all $\Phi\in F(\Rd;X)$, $\Phi=0$ if $\supp(\widehat{\Phi}\,)\subseteq\{0\}$, then we identify $F(\Rd;X)$ with its image in $\dot{\Sw}'(\Rd;X)$ under the quotient map $\Sw'(\Rd;X)\to \dot{\Sw}'(\Rd;X)$. This is the case if $F(\Rd;X)$ is a Besov space or an $L^{p}$-space for some $p\in[1,\infty]$.

Let $\psi\in\Sw(\R)$ be as in \eqref{Schwartz function defining Besov spaces}, and for $k\in\Z$ let
\begin{align}\label{eq:dyadicJ}
J_{k}:=\{\xi\in\Rd\mid 2^{k-1}\leq \abs{\xi}\leq 2^{k+1}\}.
\end{align}
Let $\psi_{k}\in\Sw(\Rd)$ be such that $\widehat{\psi}(\xi)=\psi(2^{-k}\abs{\xi})$ for $\xi\in\Rd$. Throughout this article we will keep the sequence $(\psi_{k})_{k\in\Z}$ fixed.

Let $s\in\R$ and $p,v\in[1,\infty]$. The \emph{homogeneous Besov space} $\dot{\Be}^{s}_{p,v}(\Rd;X)$ consists of all $f\in\dot{\Sw}'(\Rd;X)$ such that $\psi_{k}\ast f\in\Ellp(\Rd;X)$ for each $k\in\Z$ and
\begin{align*}
\|f\|_{\dot{\Be}^{s}_{p,v}(\Rd;X)}:=\Big\|\Big(2^{ks}\|\psi_{k}\ast f\|_{\Ellp(\Rd;X)}\Big)_{k\in\Z}\Big\|_{\ell^{v}(\Z)}<\infty,
\end{align*}
endowed with the norm $\|\cdot\|_{\dot{\Be}^{s}_{p,v}(\Rd;X)}$. Then $\dot{\Be}^{s}_{p,v}(\Rd;X)$ is a Banach space and
\begin{align*}
\dot{\Sw}(\Rd;X)\subseteq\dot{\Be}^{s}_{p,v}(\Rd;X)\subseteq\dot{\Sw}'(\Rd;X)
\end{align*}
continuously, where the first embedding has dense range if $p,v\in[1,\infty)$. Again a different choice of $\psi$ would lead to an equivalent norm on $\dot{\Be}^{s}_{p,v}(\Rd;X)$. Finally, the first embedding in \eqref{embedding between Besov spaces} is clearly still true in the homogeneous setting.

\subsection{Spaces of $\gamma$-radonifying operators}\label{gamma-spaces}

In this section we present some of the basics of the theory of $\gamma$-radonifying operators and $\gamma$-boundedness (see \cite{HNVW2}, \cite{Kalton-Weis04}, \cite{vanNeerven10}).

Let $H$ be a Hilbert space and $X$ a Banach space. An operator
$T\in\La(H,X)$ is \emph{$\gamma$-summing} if
\begin{align*}
\norm{T}_{\gamma_{\infty}(H,X)}:=\sup_{F}\Big(\mathbb{E}\Big\|\sum_{h\in
F}\gamma_{h} Th \Big\|_{X}^{2}\Big)^{1/2}<\infty,
\end{align*}
where the supremum is taken over all finite orthonormal
systems $F\subseteq H$ and $(\gamma_{h})_{h\in F}$ is a Gaussian sequence. Let $\gamma_{\infty}(H,X)$ be the space of all $\gamma$-summing operators in $\La(H,X)$, endowed with the norm $\norm{\cdot}_{\gamma_{\infty}(H,X)}$. Then the space of finite-rank operators $H\otimes X\subseteq\La(H,X)$ is contained in $\gamma_{\infty}(H,X)$, and
\begin{align}\label{gamma-norm finite-rank operators}
\Big\|\sum_{k=1}^{n}h_{k}\otimes x_{k}\Big\|_{\gamma_{\infty}(H,X)}=\Big(\mathbb{E}\Big\|\sum_{k=1}^{n}\gamma_{k}x_{k}\Big\|_{X}^{2}\Big)^{1/2}
\end{align}
for all $n\in\N$, $h_{1},\ldots, h_{n}\subseteq H$ orthonormal and $x_{1},\ldots, x_{n}\subseteq X$. We let $\gamma(H,X)$ be the closure in $\gamma_{\infty}(H,X)$ of the finite-rank operators $H\otimes X\subseteq\La(H,X)$, and we call $\gamma(H,X)$ the space of \emph{$\gamma$-radonifying operators}. If $H$ is separable with orthonormal basis $(h_{k})_{k\in\N}\subseteq H$ and $(\gamma_{k})_{k\in\N}$ is a Gaussian sequence, then by \cite[Proposition 3.19]{vanNeerven10} a $T\in\La(H,X)$ is $\gamma$-summing if and only if $\sup_{n\in\N}\mathbb{E}\|\sum_{k=1}^{n}\gamma_{k}T(h_{k})\|^{2}_{X}<\infty$, in which case
\begin{align}\label{testing_orthonormal_basis}
\|T\|_{\gamma_{\infty}(H,X)}=\sup_{n\in\N}\Big(\mathbb{E}\Big\|\sum_{k=1}^{n}\gamma_{k}T(h_{k})\Big\|^{2}_{X}\Big)^{1/2}.
\end{align}
Moreover, $T\in \gamma(H,X)$ if and only if $\sum_{k=1}^{\infty}\gamma_{k}T(h_{k})$ converges in $\Ell^2(\Omega;X)$, in which case \eqref{testing_orthonormal_basis} still holds and also equals the $\Ell^2(\Omega;X)$-norm of the series.

The following lemma introduces a useful property of the spaces of $\gamma$-summing and $\gamma$-radonifying operators, the \emph{ideal property}. For a proof see \cite[Theorem 6.2]{vanNeerven10}.

\begin{lemma}\label{ideal property gamma}
Let $H$, $K$ be Hilbert spaces and $X$, $Y$ Banach spaces. Let $R\in \La(X,Y)$, $S\in\gamma_{\infty}(H,X)$ and $T\in\La(K,H)$. Then $RST\in\gamma_{\infty}(K,Y)$ with
\begin{align*}
\norm{RST}_{\gamma_{\infty}(K,Y)}\leq \norm{R}_{\La(X,Y)}\norm{S}_{\gamma_{\infty}(H,X)}\norm{T}_{\La(K,H)}.
\end{align*}
If $S\in\gamma(H,X)$ then $RST\in\gamma(K,Y)$.
\end{lemma}

For a measure space $(\Omega,\mu)$, let $\gamma(\Omega;X)$ (resp.~$\gamma_{\infty}(\Omega;X)$) be the space of all strongly measurable functions $f:\Omega\rightarrow X$ such that $\lb f, x^*\rb\in \Elltwo(\Omega)$ for all $x^*\in X^*$ and for which the operator $J_{f}\in\La(\Ell^{2}(\Omega),X)$, given by
\begin{align}\label{definition integration operator}
J_{f}(g):=\int_{\Omega}g f\,
\,\ud\mu\quad\quad\quad(g\in \Ell^{2}(\Omega)),
\end{align}
is $\gamma$-radonifying (resp.~$\gamma$-summing). Endow $\gamma(\Omega;X)$ and $\gamma_{\infty}(\Omega;X)$ with the norm $\norm{f}_{\gamma(\Omega;X)}:=\norm{J_{f}}_{\gamma_{\infty}(\Ell^{2}(\Omega),X)}$. We will identify elements $f\otimes x\in \Ell^{2}(\Omega)\otimes X$ with the corresponding functions $g\in\gamma(\Omega;X)$ given by $g(\omega):=f(\omega)x$ for $\omega\in\Omega$. If $\Omega=\Rd$ then the following continuous embeddings hold (see \cite[Theorem 1.1]{KvNV08}):
\begin{align}\label{embedding Schwartz functions and gamma-functions}
\Sw(\Rd;X)\hookrightarrow \gamma(\R^d;X)\hookrightarrow \Sw'(\Rd;X)
\end{align}
Each of these embeddings has dense range and the same holds with $\Sw$ and $\Sw'$ replaced by $\dot{\Sw}$ and $\dot{\Sw}'$. In fact, since any $f\in\gamma_{\infty}(\Rd;X)$ with $\supp(\widehat{f}\,)\subseteq\{0\}$ satisfies $f=0$, we may view $\gamma_{\infty}(\Rd;X)$ and $\gamma(\Rd;X)$ as subsets of $\dot{\Sw}'(\Rd;X)$ through the quotient map $\Sw'(\Rd;X)\to \dot{\Sw}'(\Rd;X)$, and we will do so throughout. Note also that $\widehat{f}\in\gamma(\Rd;X)$ and
\begin{align}\label{Fourier transform on gamma}
\|f\|_{\gamma(\Rd;X)}=\|\widehat{f}\|_{\gamma(\Rd;X)}
\end{align}
for each $f\in\gamma(\Rd;X)$, by Lemma \ref{ideal property gamma}.

Let $X$ and $Y$ be Banach spaces. A collection $\mathcal{T}\subseteq \La(X,Y)$ is said to be \emph{$\gamma$-bounded} if there is a constant $C\geq 0$ such that
\begin{align}\label{gamma-boundedness}
\Big(\mathbb{E}\Big\|\sum_{k=1}^{n}\gamma_{k}T_{k}x_{k}\Big\|_{Y}^{2}\Big)^{1/2}\leq
C\Big(\mathbb{E}\Big\|\sum_{k=1}^{n}\gamma_{k}x_{k}\Big\|_{X}^{2}\Big)^{1/2}
\end{align}
for all $n\in\N$, $T_{1},\ldots, T_{n}\in\mathcal{T}$, $x_{1},\ldots, x_{n}\in X$ and each Gaussian sequence $(\gamma_{k})_{k=1}^{n}$. The smallest such $C$ is the \emph{$\gamma$-bound} of $\mathcal{T}$ and will be denoted by $\gamma(\mathcal{T})$. Often we simply write $\gamma(\mathcal{T})<\infty$ to indicate that a collection $\mathcal{T}\subseteq\La(X,Y)$ is $\gamma$-bounded. For example, when we write $(\gamma(\{m_{k}\}))_{k}\in\ell^{\infty}$, where $m_{k}\subseteq\La(X,Y)$ for each $k\in\N_{0}$, then we implicitly mean that $m_{k}\subseteq\La(X,Y)$ is $\gamma$-bounded for each $k\in\N_{0}$. By the Kahane-Khintchine inequalities, the $\Elltwo$-norm in \eqref{gamma-boundedness} may be replaced by an $\Ellp$-norm for each $p\in[1,\infty)$.

Each $\gamma$-bounded collection $\mathcal{T}$ is uniformly bounded by $\gamma(\mathcal{T})$. Conversely, each uniformly bounded collection is $\gamma$-bounded if and only if $X$ has cotype $2$ and $Y$ has type $2$ (see \cite{Arendt-Bu02}).
If $\mathcal{T}\subseteq\La(X,Y)$ is $\gamma$-bounded and $\lambda\in[0,\infty)$, then the Kahane contraction principle implies that the strong operator topology closure of $\{z T\mid z\in\C, \abs{z}\leq \lambda, T\in\mathcal{T}\}\subseteq\La(X,Y)$ is $\gamma$-bounded, and
\begin{align}\label{Kahane contraction principle}
\gamma\Big(\overline{\{z T\mid z\in\C, \abs{z}\leq \lambda, T\in\mathcal{T}\}}^{\text{SOT}}\Big)\leq \lambda\gamma(\mathcal{T}).
\end{align}

If one replaces the Gaussian random variables in \eqref{gamma-boundedness} by Rademacher variables, one obtains an \emph{$R$-bounded} collection $\mathcal{T}\subseteq\La(X,Y)$. Each $\gamma$-bounded collection is $R$-bounded, and the converse holds if and only if $X$ has finite cotype (see \cite[Theorem 1.1]{KVW14}). However, the minimal constant $C$ in \eqref{gamma-boundedness} may depend on whether one considers $\gamma$-boundedness or $R$-boundedness. In this article we work with $\gamma$-boundedness since we will obtain results for spaces which do not have finite cotype. Moreover, the notion of $\gamma$-boundedness occurs naturally in the context of $\gamma$-radonifying operators, as evidenced by the \emph{$\gamma$-Multiplier Theorem} of \cite[Proposition 4.11]{Kalton-Weis04} (see also \cite[Theorem 5.2]{vanNeerven10}).

\begin{theorem}[$\gamma$-Multiplier Theorem]\label{gamma-multiplier theorem}
Let $(\Omega,\Sigma,\mu)$ be a measure space, $X$ and $Y$ Banach spaces and $m:\Omega\to \La(X,Y)$ an $X$-strongly measurable mapping such that $\left\{m(\omega)\mid \omega\in\Omega\right\}\subseteq\La(X,Y)$ is $\gamma$-bounded. Then $mf\in\gamma_{\infty}(\Omega;Y)$ for all $f\in\gamma(\Omega;X)$, and
\begin{align*}
\|mf\|_{\gamma(\Omega;Y)}\leq \gamma(\{m(\omega)\mid \omega\in\Omega\})\,\|f\|_{\gamma(\Omega;X)}.
\end{align*}
Moreover, if there exists a dense subset $X_{0}\subseteq X$ such that $\one_{A} m(\cdot)x\in \gamma(\Omega;Y)$ for all $x\in X_0$ and $A\in \Sigma$ with $\mu(A)<\infty$, then $mf\in\gamma(\Omega;Y)$ for all $f\in \gamma(\Omega;X)$.
\end{theorem}

\section{Fourier multipliers\label{sec:Fourierm}}

In this section we introduce operator-valued Fourier multipliers on vector-valued function spaces. First we consider their basic properties and prove an approximation lemma which we will use later on, and then we discuss some of the specifics of Fourier multiplier operators on vector-valued Besov spaces.

\subsection{Basic properties of multipliers}\label{subsec:basic properties}

Throughout this section we fix $d\in\N$ and Banach spaces $X$ and $Y$. An $X$-strongly measurable $m:\Rd\to\La(X,Y)$ is \emph{of moderate growth at infinity} if there are a constant $\alpha\in(0,\infty)$ and a $g\in \Ell^1(\R^d)$ such that
\begin{align*}
(1+\abs{\xi})^{-\alpha} \|m(\xi)\|_{\La(X,Y)} \leq g(\xi) \qquad(\xi\in\Rd).
\end{align*}
For such an $m$ we let
\begin{align*}
T_{m}(f):=\F^{-1}(m\cdot\widehat{f}\,)\in\Sw'(\Rd;Y)\qquad(f\in\Sw(\Rd;X)).
\end{align*}
We call $T_{m}:\Sw(\Rd;X)\to\Sw'(\Rd;Y)$ the \emph{Fourier multiplier operator} associated with $m$ and we call $m$ the multiplier or the \emph{symbol} of $T_{m}$.

Let $F(\Rd;X)$ and $G(\Rd;Y)$ be function spaces such that $\Sw(\Rd;X)\cap F(\Rd;X)\subseteq F(\Rd;X)$ is dense and such that $G(\Rd;Y)\subseteq\Sw'(\Rd;Y)$. Then $m$ is a bounded \emph{$(F(\Rd;X), G(\Rd;Y))$-Fourier multiplier} if there is a constant $C\in(0,\infty)$ such that $T_{m}(f)\in G(\Rd;Y)$ and
\begin{align*}
\|T_{m}(f)\|_{G(\Rd;Y)}\leq C\|f\|_{F(\Rd;X)}
\end{align*}
for all $f\in\Sw(\Rd;X)\cap F(\Rd;X)$. In this case $T_{m}$ extends uniquely to a bounded operator from $F(\Rd;X)$ to $G(\Rd;Y)$ which will be denoted by $\widetilde{T_{m}}$, or just by $T_m$ when there is no danger of confusion. If $X=Y$ and $F(\Rd;X)=G(\Rd;Y)$ then we say that $m$ is an $F(\Rd;X)$-Fourier multiplier.

We will consider $(F(\Rd;X),G(\Rd;X))$-Fourier multipliers in the cases where $F(\Rd;X)=L^{p}(\Rd;X)$ or $F(\Rd;X)=\Be^{s}_{p,v}(\Rd;X)$ for $s\in\R$ and $p,v\in[1,\infty)$, and $G(\Rd;Y)=L^{q}(\Rd;Y)$ or $G(\Rd;Y)=\Be^{t}_{q,w}(\Rd;Y)$ for $t\in\R$ and $q,w\in[1,\infty]$. We shall also consider the case where $F(\Rd;X)=\Ell^{p}_{\Omega}(\Rd;X)$ and $G(\Rd;Y)=\Ell^{q}_{\Omega}(\Rd;Y)$ for certain $\Omega\subseteq\Rd$, as in \eqref{analytic Lp-functions}.

We also consider Fourier multipliers on homogeneous function spaces. Let $X$ and $Y$ be Banach spaces and let $m:\Rd\setminus\{0\}\to\La(X,Y)$ be $X$-strongly measurable. We say that $m:\Rd\setminus\{0\}\to\La(X,Y)$ is of \emph{moderate growth at zero and infinity} if there exist a constant $\alpha\in(0,\infty)$ and a $g\in \Ell^1(\R^d)$ such that
\begin{align*}
\abs{\xi}^{\alpha}(1+\abs{\xi})^{-2\alpha} \|m(\xi)\|_{\La(X,Y)} \leq g(\xi) \qquad(\xi\in\Rd).
\end{align*}
For such an $m$, let $\dot{T}_{m}:\dot{\Sw}(\Rd;X)\to\Sw'(\Rd;Y)$ be given by
\begin{align*}
\dot{T}_{m}(f):=\F^{-1}(m\cdot\widehat{f}\,)\qquad(f\in\dot{\Sw}(\Rd;X)),
\end{align*}
where $\dot{T}_{m}(f)\in\Sw'(\Rd;Y)$ is well-defined by definition of $\dot{\Sw}(\Rd;X)$. We use similar terminology as before to discuss the boundedness of $\dot{T}_{m}$ with respect to various homogeneous function spaces. When considering bounded $\dot{T}_{m}:\Ellp(\Rd;X)\to\Ellq(\Rd;Y)$ we will sometimes simply write $T_{m}=\dot{T}_{m}$.

In later sections we use that the space $\Ell^{p}_{\Omega}(\Rd;X)\cap \Sw(\Rd;X)=\Sw_{\Omega}(\Rd;X)$ is dense in $\Ell^{p}_{\Omega}(\Rd;X)$ for a large class of $\Omega\subseteq\Rd$. A similar result will be needed for $\gamma$-spaces.
For $\Omega\subseteq\Rd$ define
\begin{align}\label{gamma-functions with compact support}
\gamma_{\Omega}(\Rd;X):=\{f\in\gamma(\Rd;X)\mid\supp(\widehat{f}\,)\subseteq\Omega\}.
\end{align}
In order to state such a denseness result we need the following definition.
A bounded open set $\Omega\subseteq\Rd$ is said to have the \emph{segment property} if there exist $N\in\N$, open balls $B_{1},\ldots, B_{N}$ in $\Rd$ and $y_{1},\ldots, y_{N}\in\Rd$ such that $\overline{\Omega}\subseteq \cup_{k=1}^{N}B_{k}$ and
\begin{align}\label{definition_segment_property}
(\overline{\Omega}\cap B_{k})+ty_{k}\subseteq \Omega\qquad(k\in\{1,\ldots, N\},\, t\in(0,1]).
\end{align}
Note that sets of the form $(a,b)^{d}$ for $a,b\in\R$ with $a<b$, and the interior of the annuli $I_{k}$ and $J_{k}$ from \eqref{dyadic annuli} and \eqref{eq:dyadicJ} have the segment property.

The following result is known in the scalar case, cf.~\cite[Section 1.4.3]{Triebel10}.
For the reader's convenience we include a proof and additionally we consider the case of $\gamma$-spaces as well.

\begin{lemma}\label{lem:approxsmooth}
Let $X$ be a Banach space, $p\in[1,\infty)$ and let $\Omega\subseteq\Rd$ have the segment property. Then $\Sw_{\Omega}(\Rd;X)$ is dense in $\Ell^{p}_{\overline{\Omega}}(\Rd;X)$ and in $\gamma_{\overline{\Omega}}(\Rd;X)$.
\end{lemma}
\begin{proof}
Let $N\in\N$, open balls $(B_{k})_{k=1}^{N}$ and $(y_{k})_{k=1}^{N}\subseteq\Rd$ be such that $\overline{\Omega}\subseteq \bigcup_{k=1}^{N} B_{k}$ and such that \eqref{definition_segment_property} holds. Let $(\chi_{k})_{k=1}^{N}\subseteq\Sw(\Rd)$ be such that $\sum_{k=1}^{N} \widehat{\chi_{k}} = 1$ on $\overline{\Omega}$ and such that $0\leq \widehat{\chi_{k}}\leq 1$ and $\supp(\widehat{\chi_{k}})\subseteq B_{k}$ for all $k\in\{1,\ldots, N\}$.

Let $f\in \Ell^{p}_{\overline{\Omega}}(\Rd;X)$ and let $f_{k}:=\chi_{k}\ast f\in \Ellp(\Rd;X)$ for all $k\in\{1,\ldots, N\}$.
Then
\begin{align*}
\supp(\F (\ue^{-2\pi \ui t y_k \cdot} f_k))  =  \supp(\widehat{f_{k}}(\cdot + t y_k))\subseteq \Omega
\end{align*}
for all $k\in\{1,\ldots, N\}$ and $t\in (0,1]$. Moreover, by the dominated convergence theorem, $\lim_{t\downarrow 0} \ue^{-2\pi \ui t y_k \cdot} f_k = f_k$ in $\Ellp(\Rd;X)$. Let $\varepsilon>0$ and let $t\in (0,1]$ be such that $g_{k} := \ue^{-2\pi\ui t y_k \cdot} f_k$ satisfies
\begin{align*}
\|g_k  - f_k\|_{\Ellp(\Rd;X)}<\varepsilon \ \ \ \text{for all $k\in \{1, \ldots, N\}$.}
\end{align*}
Let $\varphi\in \Sw(\R^d)$ be such that $\ph(0) = 1$ and $\supp (\widehat{\varphi})\subseteq [-1,1]^d$. Let $g_{k,n}(t) := \varphi\left(\tfrac{t}{n}\right)\! g_k(t)$ for $k\in\{1,\ldots,N\}$, $n\in\N$ and $t\in\Rd$. Then $\widehat{g_{k,n}} = n^d \widehat{\varphi}(n \cdot)*\widehat{g_{k}}\in \Sw(\Rd;X)$ (see \cite[Theorem 2.3.20]{Grafakos08}) and, for all $n\in\N$ large enough, $\supp(\widehat{g_{k,n}})\subseteq \Omega$. Moreover, $g_{k,n}\to g_k$ in $\Ellp(\Rd;X)$ as $n\to \infty$, by the dominated convergence theorem. Fixing $n\in\N$ large enough we obtain
\begin{align*}
g_{k,n}\in \Sw_{\Omega}(\Rd;X) \ \ \text{and} \ \ \|g_{k,n} - g_k\|_{\Ellp(\R^d;X)}<\varepsilon
\end{align*}
for all $k\in\{1,\ldots, N\}$. Let $g := \sum_{k=1}^{N} g_{k,n}\in \Sw_{\Omega}(\Rd;X)$. Combining all these estimates and using that $f=\sum_{k=1}^{N}f_{k}$, we obtain
\begin{align*}
\|f- g\|_{p} \leq \sum_{k=1}^{N} \|f_{k} - g_{k,n}\|_{p}\leq  \sum_{k=1}^{N} \|f_k - g_{k}\|_{p} + \sum_{k=1}^{N}\|g_k - g_{k,n}\|_{p}<2N \varepsilon.
\end{align*}
Letting $\eps$ decrease to zero now yields the first statement.

Next let $f\in \gamma_{\overline{\Omega}}(\Rd;X)$ and let $\eps>0$. Let $(h_k)_{k\in\N}\subseteq
\Ell^2_{\overline{\Omega}}(\Rd)$ be an orthonormal basis for $\Ell^2_{\overline{\Omega}}(\Rd)$, and for $n\in\N$ let $g_{n} := \sum_{k=1}^n h_{k} \otimes J_{f}(h_k)\in \gamma_{\overline{\Omega}}(\Rd;X)$. Then, by \eqref{testing_orthonormal_basis},
\begin{align*}
\|f-g_{n}\|_{\gamma(\Rd;X)}&=\|J_{f}-J_{g_{n}}\|_{\gamma_{\infty}(\Ell^{2}(\Rd),X)}=\|J_{f}-J_{g_{n}}\|_{\gamma_{\infty}(\Ell^{2}_{\overline{\Omega}}(\Rd),X)}\\
&=\sup_{N\in\N}\Big(\mathbb{E}\Big\|\sum_{k=1}^{N}\gamma_{k}(J_{f}(h_{k})-J_{g}(h_{k}))\Big\|_{X}^{2}\Big)^{1/2}\\
&=\sup_{N\geq n}\Big(\mathbb{E}\Big\|\sum_{k=n}^{N}\gamma_{k}J_{f}(h_{k})\Big\|_{X}^{2}\Big)^{1/2}\to 0
\end{align*}
as $n\to\infty$. Hence it follows that we can find $n\in\N$ such that $\|f-g_{n}\|_{\gamma(\Rd;X)}<\varepsilon$. Since $(h_{k})_{k\in\N}\subseteq \Ell^2_{\overline{\Omega}}(\Rd)$ it follows from
the previous part of the proof that for each $k\in\N$ there exist $\zeta_k\in \Sw_{\Omega}(\Rd)$ such that $\|h_k - \zeta_k\|_{2}<\tfrac{\varepsilon}{n}$. Let $g:= \sum_{k=1}^n \zeta_{k}\otimes J_{f}(h_{k}) \in\Sw_{\Omega}(\Rd;X)$. Then, by \eqref{gamma-norm finite-rank operators},
\begin{align*}
\|f- g\|_{\gamma(\Rd;X)} &\leq \|f- g_n\|_{\gamma(\Rd;X)} + \|g_n- g\|_{\gamma(\Rd;X)}\\
&<\eps+\|J_{g_{n}}-J_{g}\|_{\gamma_{\infty}(\Ell^{2}(\Rd),X)}\\
&\leq \eps+\sum_{k=1}^{n}\|(h_{k}-\zeta_{k})\otimes J_{f}(h_{k})\|_{\gamma_{\infty}(\Ell^{2}(\Rd),X)}\\
&=\varepsilon + \sum_{k=1}^n \|J_{f}(h_{k})\|_{X}\, \|h_k - \zeta_k\|_{\Ell^{2}(\Rd)}<\varepsilon(1+\|J_{f}\|_{\La(\Ell^{2}(\Rd),X)}).
\end{align*}
Letting $\eps$ tend to zero concludes the proof.
\end{proof}

\subsection{Fourier multipliers on Besov spaces in an abstract setting}\label{subsec:Fourier multipliers on Besov spaces}

Fix $p\in[1,\infty)$ and $q\in[1,\infty]$. For $k\in\N_{0}$ recall the definition of $\ph_{k}\in\Sw(\Rd)$ from \eqref{Littlewood-Paley functions} and $I_{k}\subseteq\Rd$ from \eqref{dyadic annuli}.

Below we consider $X$-strongly measurable $m:\Rd\to\La(X,Y)$ of moderate growth at infinity with the following property. There exist $\beta\in\R$, $u\in[1,\infty]$ and $(c_{k})_{k\in\N_{0}}\in\ell^{u}$ such that $m$ is an $(\Ell^{p}_{I_{k}}(\Rd;X),\Ell^{q}_{I_{k}}(\Rd;Y))$-Fourier multiplier for each $k\in\N_{0}$, and
\begin{align}\label{general estimate on intervals}
\|T_{m}(f)\|_{\Ellq(\Rd;Y)}\leq 2^{k\beta}c_{k}\|f\|_{\Ellp(\Rd;X)}\qquad (f\in\Sw_{I_{k}}(\Rd;X)).
\end{align}
We will show how such an estimate can be used to obtain a Fourier multiplier result in the Besov scale.

Let $s\in\R$ and $v,w\in[1,\infty]$ be such that $\frac{1}{w}\leq \frac{1}{z} = \frac{1}{u}+\frac{1}{v}$. Note that
\begin{align}\label{convolution with Schwartz functions}
\ph\ast T_{m}(f)=\F^{-1}(\widehat{\ph}\cdot m\widehat{f}\,)=\F^{-1}(m\cdot\widehat{\ph}\widehat{f}\,)=T_{m}(\ph\ast f)
\end{align}
for all $f\in\Sw(\Rd;X)$ and $\ph\in\Sw(\Rd)$. Therefore, using the contractive inclusion $\ell^{z}\subseteq\ell^{w}$ and H\"{o}lder's inequality, we obtain
\begin{equation}\label{eq:estholderBesov}
\begin{aligned}
\|T_{m}(f)\|_{\Be^{s-\beta}_{q,w}(\Rd;Y)}&=\Big\|\Big(2^{k(s-\beta)}\|\ph_{k}\ast T_{m}(f)\|_{\Ellq(\Rd;Y)}\Big)_{k}\Big\|_{\ell^{w}}\\
&\leq \Big\|\Big(2^{k(s-\beta)}\|T_{m}(\ph_{k}\ast f)\|_{\Ellq(\Rd;Y)}\Big)_{k}\Big\|_{\ell^{z}}\\
&\leq \Big\|\Big(2^{ks}c_{k}\|\ph_{k}\ast f\|_{\Ellp(\Rd;X)}\Big)_{k}\Big\|_{\ell^{z}}
\\ & \leq \|(c_{k})_{k}\|_{\ell^{u}}\|f\|_{\Be^{s}_{p,v}(\Rd;X)}
\end{aligned}
\end{equation}
for all $f\in\Sw(\Rd;X)$. Since $\Sw(\Rd;X)$ is dense in $\Be^{s}_{p,v}(\Rd;X)$ for $v\in[1,\infty)$, \eqref{eq:estholderBesov} implies that $T_{m}$ is a bounded $(\Be^{s}_{p,v}(\Rd;X),\Be^{s-\beta}_{q,w}(\Rd;Y))$-Fourier multiplier if $v<\infty$. In the remainder of this section we discuss a method that will allow us to deal with all $v\in[1,\infty]$ simultaneously.

For $n\in\N_{0}$ denote by $T_{m}^{(n)}\in\La(\Ell^{p}_{I_{n}}(\Rd;X),\Ell^{q}_{I_{n}}(\Rd;Y))$ the unique bounded extension of $T_{m}\!\restriction_{\Sw_{I_{n}}(\Rd;X)}$, which exists by \eqref{general estimate on intervals} and Lemma \ref{lem:approxsmooth} and which has norm $\|T_{m}^{(n)}\|\leq 2^{n\beta}c_{n}$. For later use we note, as follows easily from \eqref{convolution with Schwartz functions}, that
\begin{align}\label{extension commutes with convolution}
\ph\ast T_{m}^{(n)}(g)=T_{m}^{(n)}(\ph\ast g)
\end{align}
for all $\ph\in\Sw(\Rd)$, $n\in\N_{0}$ and $g\in\Ell^{p}_{I_{n}}(\Rd;X)$. Now define, for $s\in\R$ and $v\in[1,\infty]$,
\begin{align}\label{definition extension}
\widetilde{T_{m}}(f):=\sum_{n=0}^{\infty}T_{m}^{(n)}(\ph_{n}\ast f)\qquad(f\in\Be^{s}_{p,v}(\Rd;X))
\end{align}
as a convergent series in $\Sw'(\Rd;Y)$. The following proposition shows in particular that this is well-defined.
The assumption that $m$ is of moderate growth at infinity is only made to ensure that $T_{m}:\Sw(\Rd;X)\to\Sw'(\Rd;Y)$ is well-defined.

\begin{proposition}\label{prop:multipliers_on_Besov_spaces}
Let $X$ and $Y$ be Banach spaces, $p\in[1,\infty)$ and $q\in[1,\infty)$. Let $m:\Rd\to\La(X,Y)$ be an $X$-strongly measurable map of moderate growth at infinity such that \eqref{general estimate on intervals} holds for all $k\in\N_{0}$ and for some $\beta\in\R$, $u\in[1,\infty]$ and $(c_{k})_{k\in\N_{0}}\in\ell^{u}$.
Then \eqref{definition extension} defines an extension of $T_{m}$ to a bounded linear map $\widetilde{T}_m$
from $\Be^{s}_{p,v}(\Rd;X)$ to $\Be^{s-\beta}_{q,w}(\Rd;Y)$ of norm $\|T_m\| \leq \|(c_{k})_{k}\|_{\ell^{u}}$
for all $s\in\R$ and all $v,w\in[1,\infty]$ with $\frac{1}{w}\leq \frac{1}{u}+\frac{1}{v}$.
\end{proposition}
The extension of $T_m$ is unique for $v<\infty$, by the density of $\Sw(\R^d;X)$ in $\Be^{s}_{p,v}(\Rd;X)$. The uniqueness of the extension for $v= \infty$ is discussed in Remark \ref{rem:unique}.

\begin{proof}
Let $s\in\R$, $v\in[1,\infty]$ and $f\in\Be^{s}_{p,v}(\Rd;X)$. Then, using H\"{o}lder's inequality and the contractive inclusion $\ell^{u}\subseteq\ell^{\infty}$, there exists a constant $C\in(0,\infty)$ such that, for each $\ph\in\Sw(\Rd)$ and with the obvious modification for $v=1$,
\begin{align*}
&\sum_{n=0}^{\infty}\|\lb T_{m}^{(n)}(\ph_{n}\ast f),\ph\rb\|_{Y}=\sum_{n=0}^{\infty}\|\lb T_{m}^{(n)}(\ph_{n}\ast f),(\ph_{n-1}+\ph_{n}+\ph_{n+1})\ast\ph\rb\|_{Y}\\
\leq &\sum_{n=0}^{\infty}\|T_{m}^{(n)}(\ph_{n}\ast f)\|_{\Ellq(\Rd;Y)}\|(\ph_{n-1}+\ph_{n}+\ph_{n+1})\ast\ph\|_{\Ell^{q'}(\Rd)}\\
\leq &\sum_{n=0}^{\infty}c_{n}2^{n\beta}\|\ph_{n}\ast f\|_{\Ellp(\Rd;Y)}\|(\ph_{n-1}+\ph_{n}+\ph_{n+1})\ast\ph\|_{\Ell^{q'}(\Rd)}\\
\leq &\|(c_{n})_{n}\|_{\ell^{\infty}}\sum_{n=0}^{\infty}2^{ns}\|\ph_{n}\ast f\|_{\Ellp(\Rd;Y)}2^{n(\beta-s)}\|(\ph_{n-1}+\ph_{n}+\ph_{n+1})\ast\ph\|_{\Ell^{q'}(\Rd)}\\
\leq &\|(c_{n})_{n}\|_{\ell^{\infty}}\|f\|_{\Be^{s}_{p,v}(\Rd;X)}\Big(\sum_{n=0}^{\infty}2^{n(\beta-s)v'}\|(\ph_{n-1}+\ph_{n}+\ph_{n+1})\ast\ph\|_{\Ell^{q'}(\Rd)}^{v'}\Big)^{1/v'}\\
\leq &\|(c_{n})_{n}\|_{\ell^{\infty}}\|f\|_{\Be^{s}_{p,v}(\Rd;X)}\Big(\sum_{n=0}^{\infty}\Big(\sum_{k=n-1}^{n+1}2^{n(\beta-s)}\|\ph_{k}\ast\ph\|_{\Ell^{q'}(\Rd)}\Big)^{v'}\Big)^{1/v'}\\
\leq &C\|f\|_{\Be^{s}_{p,v}(\Rd;X)}\Big(\sum_{n=0}^{\infty}2^{n(\beta-s)v'}\|\ph_{n}\ast\ph\|_{\Ell^{q'}(\Rd)}^{v'}\Big)^{1/v'}=C\|f\|_{\Be^{s}_{p,v}(\Rd;X)}\|\ph\|_{\Be^{\beta-s}_{q',v'}(\Rd)}.
\end{align*}
Since $\Sw(\Rd)\subseteq\Be^{\beta-s}_{q',v'}(\Rd)$ continuously, $\widetilde{T}_{m}(f)=\sum_{n=0}^{\infty}T_{m}^{(n)}(\ph_{n}\ast f)$ converges as a limit in $\Sw'(\Rd;Y)$.

Now let $w\in[1,\infty]$ be such that $\tfrac{1}{w}\leq \tfrac{1}{u}+\tfrac{1}{v}$. We claim that
\begin{align*}
\ph_{k}\ast \widetilde{T_{m}}(f)=T_{m}^{(k)}(\ph_{k}\ast f)
\end{align*}
for each $k\in\N_{0}$. Indeed, by Lemma \ref{lem:approxsmooth} we can find $(f_{j})_{j\in\N}\subseteq\Sw_{I_{k}}(\Rd;X)$ such that $f_{j}\to \ph_{k}\ast f$ in $\Ellp(\Rd;X)$ as $j\to \infty$. Note that $\ph_{k}\ast \ph_{n}=0$ if $n\notin \{k-1,k,k+1\}$ and that $\sum_{n=k-1}^{k+1}\ph_{n}\ast g=g$ if $g\in\Sw'(\Rd;Y)$ is such that $\supp(\widehat{g})\subseteq I_{k}$. Therefore, using \eqref{extension commutes with convolution} and arguing in $\Sw'(\Rd;Y)$, we find
\begin{align*}
\ph_{k}\ast \widetilde{T_{m}}(f)&=\ph_{k}\ast\sum_{n=0}^{\infty}T_{m}^{(n)}(\ph_{n}\ast f)=\sum_{n=0}^{\infty}\ph_{k}\ast T_{m}^{(n)}(\ph_{n}\ast f)\\
&=\sum_{n=k-1}^{k+1}T_{m}^{(n)}(\ph_{n}\ast \ph_{k}\ast f)=\sum_{n=k-1}^{k+1}\lim_{j\to\infty}T_{m}(\ph_{n}\ast f_{j})\\
&=\lim_{j\to\infty}T_{m}\Big(\sum_{n=k-1}^{k+1}\ph_{n}\ast f_{j}\Big)=\lim_{j\to\infty}T_{m}(f_{j})=T_{m}^{(k)}(\ph_{k}\ast f),
\end{align*}
as claimed. Now the required norm bound for $\widetilde{T_m}$ follows as in \eqref{eq:estholderBesov}.

To see that $\widetilde{T_{m}}$ extends $T_{m}$, let $f\in\Sw(\Rd;X)$. Then, arguing in $\Sw'(\Rd;Y)$,
\begin{align*}
\F(\widetilde{T_{m}}(f))&=\F\Big(\sum_{n=0}^{\infty}T_{m}^{(n)}(\ph_{n}\ast f)\Big)=\sum_{n=0}^{\infty}\F(T_{m}(\ph_{n}\ast f))\\
&=\sum_{n=0}^{\infty}m\cdot \widehat{\ph_{n}}\widehat{f}=\big(\sum_{n=0}^{\infty}\widehat{\ph_{n}}\Big)m\cdot \widehat{f}=m\cdot \widehat{f}=\F(T_{m}(f)),
\end{align*}
as required.
\end{proof}

\begin{remark}[Uniqueness]\label{rem:unique}
In the case $v = \infty$, the operator $\widetilde{T_m}$ from $\Be^{s}_{p,\infty}(\Rd;X)$ into $\Be^{s-\beta}_{q,w}(\Rd;Y)$ given in
Proposition \ref{prop:multipliers_on_Besov_spaces} is also bounded from $\Be^{s-1}_{p,1}(\Rd;X)$ to $\Be^{s-\beta-1}_{q,w}(\Rd;Y)$.
On this larger space (in which $\Sw(\R^d;X)$ is dense) it is the unique extension of $T_m$.
\end{remark}

\begin{remark}\label{rem:Tmbddck}
Considering functions $f$ with suitable support one sees that the boundedness of $T_m$ also implies \eqref{general estimate on intervals} with $c_k$ replaced by $K (c_{k-1} + c_{k} + c_{k+1})$, where $K$ is a constant independent of $f$ and $(c_k)_{k\geq 0}$ and $c_{-1} = 0$. In this sense the boundedness of $T_m$ is equivalent to \eqref{general estimate on intervals}.
\end{remark}

We will also consider Fourier multipliers on homogeneous Besov spaces. Let $X$ and $Y$ be Banach spaces, $p\in[1,\infty)$ and $q\in[1,\infty]$. For $k\in\Z$ recall the definition of $\psi_{k}\in\dot{\Sw}(\Rd)$ and $J_{k}\subseteq\Rd$ from Section \ref{Besov spaces}. Let $m:\Rd\setminus\{0\}\to\La(X,Y)$ be an $X$-strongly measurable map of moderate growth at zero and infinity and with the property that there exist $\beta\in\R$, $u\in[1,\infty]$ and $(c_{k})_{k\in\Z}\in\ell^{u}(\Z)$ such that, for each $k\in\Z$, $m$ is an $(\Ell^{p}_{J_{k}}(\Rd;X),\Ell^{q}_{J_{k}}(\Rd;Y))$-Fourier multiplier and
\begin{align}\label{general homogeneous estimate on intervals}
\|\dot{T}_{m}(f)\|_{\Ellq(\Rd;Y)}\leq 2^{k\beta}c_{k}\|f\|_{\Ellp(\Rd;X)}\qquad (f\in\Sw_{J_{k}}(\Rd;X)).
\end{align}
For $n\in\Z$ denote by $\dot{T}_{m}^{(n)}\in\La(\Ell^{p}_{J_{n}}(\Rd;X),\Ell^{q}_{J_{n}}(\Rd;Y))$ the unique bounded extension of $\dot{T}_{m}\!\restriction_{\Sw_{J_{n}}(\Rd;X)}$. For $s\in\R$ and $v\in[1,\infty]$, define
\begin{align}\label{definition homogeneous extension}
\overline{T_{m}}(f):=\sum_{n=-\infty}^{\infty}\dot{T}_{m}^{(n)}(\psi_{n}\ast f)\qquad(f\in\dot{\Be}^{s}_{p,v}(\Rd;X)).
\end{align}
The following proposition is proved in the same way as Proposition \ref{prop:multipliers_on_Besov_spaces}.

\begin{proposition}\label{prop:multipliers_on_homogeneous_Besov_spaces}
Let $X$ and $Y$ be Banach spaces, $p\in[1,\infty)$ and $q\in[1,\infty)$. Let $m:\Rd\setminus\{0\}\to\La(X,Y)$ be an $X$-strongly measurable map of moderate growth at zero and infinity such that \eqref{general homogeneous estimate on intervals} holds for all $k\in\Z$ and for some $\beta\in\R$, $u\in[1,\infty]$ and $(c_{k})_{k\in\Z}\in\ell^{u}(\Z)$. Then $\overline{T_{m}}$ extends $\dot{T}_{m}$ to a bounded linear operator from $\dot{\Be}^{s}_{p,v}(\Rd;X)$ into $\dot{\Be}^{s-\beta}_{q,w}(\Rd;Y))$ with
\begin{align*}
\|\overline{T_{m}}\|_{\La(\dot{\Be}^{s}_{p,v}(\Rd;X),\dot{\Be}^{s-\beta}_{q,w}(\Rd;Y))}\leq \|(c_{k})_{k}\|_{\ell^{u}(\Z)}
\end{align*}
for all $s\in\R$ and all $v,w\in[1,\infty]$ such that $\frac{1}{w}\leq \frac{1}{u}+\frac{1}{v}$.
\end{proposition}
As before, the extension $\overline{T_m}$ is unique if $v<\infty$ by the density of $\dot{\Sw}(\R^d;X)$ in $\dot{\Be}^{s}_{p,v}(\Rd;X)$.
If $v= \infty$, then one cannot argue as in Remark \ref{rem:unique}, and we leave out any uniqueness assertions in this case.

As in Remark \ref{rem:Tmbddck} one sees that the boundedness of $T_m$ is equivalent to \eqref{general homogeneous estimate on intervals}.

\begin{remark}
Using the technique of \cite[Proposition 3.4]{Rozendaal-Veraar16Fourier} one can transfer the results of Propositions \ref{prop:multipliers_on_Besov_spaces} and \ref{prop:multipliers_on_homogeneous_Besov_spaces} on $\R^d$ to the periodic setting $\T^d$. For the definition of the periodic Besov spaces we refer to \cite{Arendt-Bu-period,Triebel10}. Indeed, one can apply \eqref{general estimate on intervals} or \eqref{general homogeneous estimate on intervals} to suitable functions $f$ with compact Fourier support as in the proof of the transference result mentioned above. In particular, this yields periodic analogues of Proposition \ref{prop:BesovFourier} and Theorems \ref{main result Besov multipliers type} and \ref{main result homogeneous Besov multipliers type}. The details are left to the reader.
\end{remark}

\subsection{Fourier type setting}
In this section we use that \eqref{general estimate on intervals} holds under Fourier type conditions and apply Proposition \ref{prop:multipliers_on_Besov_spaces} to obtain a first Fourier multiplier result on Besov spaces.

A Banach space $X$ is said to have \emph{Fourier type} $p\in[1,2]$ if the Fourier transform $\F:\Ellp(\Rd;X)\to\Ellpprime(\Rd;X)$ is bounded for some (and then for all) $d\in\N$. So that our terminology is consistent with results in the rest of the paper, we say that $X$ has \emph{Fourier cotype} $q\in[2,\infty]$ if $X$ has Fourier type $q'$.

\begin{proposition}\label{prop:BesovFourier}
Let $X$ be a Banach space with Fourier type $p\in[1,2]$ and $Y$ a Banach space with Fourier cotype $q\in[2,\infty]$,
and let $r\in[1,\infty]$ be such that $\tfrac{1}{r}=\tfrac{1}{p}-\tfrac{1}{q}$.
Let $s\in\R$ and $v,w\in[1,\infty]$ be such that $\frac{1}{w}\leq \frac{1}{u}+\frac{1}{v}$.
Let $m:\Rd\to\La(X,Y)$ be an $X$-strongly measurable map such that $c_k: = \|[\xi\mapsto\norm{m(\xi)}_{\La(X,Y)}]\|_{\Ell^{r}(I_k)}<\infty$ for all $k\in \N_0$. Assume that $(c_k)_{k}\in \ell^u$ for some $u\in[1,\infty]$. Then $T_m$ extends to a bounded mapping from $\Be^{s}_{p,v}(\Rd;X)$ to $\Be^{s}_{q,w}(\Rd;Y)$ of norm $\|T_m\| \leq C\|(c_{k})_{k}\|_{\ell^{u}}$ for some $C\geq 0$ independent of $m$.
\end{proposition}
A similar result in the homogeneous setting follows from Proposition \ref{prop:multipliers_on_homogeneous_Besov_spaces}.
\begin{proof}
By the Fourier multiplier result of \cite{Rozendaal-Veraar16Fourier} under Fourier type conditions, there exists a constant $C\geq 0$ independent of $m$ such that
\begin{align*}
\|T_{m}(f)\|_{\Ellq(\Rd;Y)}\leq c_{k}\|f\|_{\Ellp(\Rd;X)}\qquad (f\in\Sw_{I_{k}}(\Rd;X))
\end{align*}
for all $k\in\N_{0}$. Consequently, the result follows from Proposition \ref{prop:multipliers_on_Besov_spaces}.
\end{proof}

\section{Fourier multipliers under type and cotype conditions}\label{sec:type and cotype assumptions}

In this section we prove our main results. We obtain Fourier multiplier theorems on Besov spaces under type and cotype conditions on the underlying spaces. As a corollary we derive a result for $(L^{p},L^{q})$-multipliers.

Let $X$ be a Banach space, $(\gamma_{n})_{n\in\N}$ a Gaussian sequence on a probability space $(\Omega,\mathbb{P})$ and let $p\in[1,2]$ and $q\in[2,\infty]$. We say that $X$ has \emph{(Gaussian) type} $p$ if there exists a constant $C\geq 0$ such that for all $m\in\N$ and all $x_{1},\ldots, x_{m}\in X$,
\begin{align}\label{type}
\Big(\mathbb{E}\Big\|\sum_{n=1}^{m}\gamma_{n}x_{n}\Big\|^{2}\Big)^{1/2}\leq C\Big(\sum_{n=1}^{m}\|x_{n}\|^{p}\Big)^{1/p}.
\end{align}
We say that $X$ has \emph{(Gaussian) cotype} $q$ if there exists a constant $C\geq 0$ such that for all $m\in\N$ and all $x_{1},\ldots, x_{m}\in X$,
\begin{align}\label{cotype}
\Big(\sum_{n=1}^{m}\|x_{n}\|^{q}\Big)^{1/q}\leq C\Big(\mathbb{E}\Big\|\sum_{n=1}^{m}\gamma_{n}x_{n}\Big\|^{2}\Big)^{1/2},
\end{align}
with the obvious modification for $q=\infty$. The minimal constants $C$ in \eqref{type} and \eqref{cotype} are called the \emph{(Gaussian) type $p$ constant} and the \emph{(Gaussian) cotype $q$ constant} and will be denoted by $\tau_{p,X}$ and $c_{q,X}$. We say that $X$ has \emph{nontrivial type} if $X$ has type $p\in(1,2]$ and \emph{finite cotype} if $X$ has cotype $q\in[2,\infty)$.

The Gaussian sequence in \eqref{type} and \eqref{cotype} is usually replaced by a \emph{Rademacher sequence}, i.e.\ a sequence $(r_{n})_{n\in\N}$ of independent identically distributed random variables with $\mathbb{P}(r_{1}=1)=\mathbb{P}(r_{1}=-1)=\frac{1}{2}$. This does not change the class of spaces under consideration, only the minimal constants in \eqref{type} and \eqref{cotype} (see \cite[Chapter 12]{DJT95}). We choose to work with Gaussian sequences because the Gaussian constants $\tau_{p,X}$ and $c_{q,X}$ occur naturally in the results in this section.

Each Banach space $X$ has type $p=1$ and cotype $q=\infty$, with $\tau_{1,X}=c_{\infty,X}=1$. If $X$ has type $p$ and cotype $q$ then $X$ has type $r$ with $\tau_{r,X}\leq \tau_{p,X}$ for all $r\in[1,p]$ and cotype $s$ with $c_{s,X}\leq c_{q,X}$ for all $s\in[q,\infty]$.

A Banach space with Fourier type $p\in[1,2]$ has type $p$ and cotype $p'$ (see \cite{HNVW2}). By a result of Bourgain a Banach space has nontrivial type if and only if it has nontrivial Fourier type (see \cite[5.6.30]{PieWen}).

For more on type and cotype see \cite{AK06}, \cite{DJT95}, \cite{HNVW2} and \cite[Section 9.2]{LT}.

\subsection{Functions with compact Fourier support under type and cotype conditions}\label{subsec:multipliers theorems on Besov spaces}

Fix $d\in\N$. For $X$ a Banach space, $\Omega\subseteq\Rd$ and $p\in[1,\infty]$, recall the definitions of $\Sw_{\Omega}(\Rd;X)$, $\Ell^{p}_{\Omega}(\Rd;X)$, and $\gamma_{\Omega}(\Rd;X)$ from \eqref{analytic Schwartz functions}, \eqref{analytic Lp-functions} and \eqref{gamma-functions with compact support}.
Note that distributions $f\in\Sw'(\Rd;X)$ with $\supp(\widehat{f}\,)\subseteq\Omega$ for some compact $\Omega\subseteq\Rd$ satisfy $f\in\Ce^{\infty}\!(\Rd;X)$ (see \cite[Theorem 2.3.21]{Grafakos08}).

The following lemma is a consequence of \cite[Lemma 2.1]{KvNV08}, which applies to $f\in \Sw_{[0,1]^d}(\Rd;\!X)$.

\begin{lemma}\label{compactly supported functions}
Let $X$ be a Banach space with type $p\in [1, 2]$ and cotype $q\in [2, \infty]$. Let $a,b\in\R$ with $a<b$. Then the following assertions hold.
\begin{enumerate}[$(1)$]
\item\label{item 1} $\Ell^{p}_{[a,b]^{d}}(\Rd;X)\subseteq\gamma_{[a,b]^{d}}(\Rd;X)$ and
\begin{align*}
\norm{f}_{\gamma(\Rd;X)}\leq \tau_{p,X}(b-a)^{d(\frac{1}{p}-\frac{1}{2})}\norm{f}_{\Ellp(\Rd;X)}
\end{align*}
for all $f\in\Ell^{p}_{[a,b]^{d}}(\Rd;X)$.
\item\label{item 2} $\gamma_{[a,b]^{d}}(\Rd;X)\subseteq\Ell^{q}_{[a,b]^{d}}(\Rd;X)$ and
\begin{align*}
\norm{f}_{\Ellq(\Rd;X)}\leq c_{q,X}(b-a)^{d(\frac{1}{2}-\frac{1}{q})}\norm{f}_{\gamma(\Rd;X)}.
\end{align*}
for all $f\in\gamma_{[a,b]^{d}}(\Rd;X)$.
\end{enumerate}
\end{lemma}
\begin{proof}
\eqref{item 1} First assume that $f\in \Sw_{[a,b]^{d}}(\Rd;X)$.
Let $g(t):=\ue^{-2\pi i \frac{b+a}{2(b-a)}t}f(\tfrac{t}{b-a})$ for $t\in\Rd$. Then $g\in\Sw_{[-\frac{1}{2},\frac{1}{2}]^{d}}(\Rd;X)$ and, by Lemma \ref{ideal property gamma},
\begin{align*}
\norm{g}_{\Ellp(\Rd;X)}=(b-a)^{\frac{d}{p}}\norm{f}_{\Ellp(\Rd;X)} \ \ \text{and} \ \ \norm{g}_{\gamma(\Rd;X)}=(b-a)^{\frac{d}{2}}\norm{f}_{\gamma(\Rd;X)}.
\end{align*}
By \cite[Lemma 2.1]{KvNV08} (note that $\F$ is normalized differently in \cite{KvNV08}),
\begin{align*}
(b-a)^{\frac{d}{2}}\norm{f}_{\gamma(\Rd;X)}&=\norm{g}_{\gamma(\Rd;X)}\leq \tau_{p,X}\norm{g}_{\Ellp(\Rd;X)}=\tau_{p,X}(b-a)^{\frac{d}{p}}\norm{f}_{\Ellp(\Rd;X)}.
\end{align*}
For a general $f\in \Ell^{p}_{[a,b]^{d}}(\Rd;X)$, let $(f_n)_{n\in\N}\subseteq\Sw_{[a,b]^{d}}(\Rd;X)$ be such that $\|f_{n}-f\|_{\Ellp(\Rd;X)}\to 0$, as in Lemma \ref{lem:approxsmooth}. By the previous estimate,
\begin{align}\label{eq:estndep}
\norm{f_n}_{\gamma(\Rd;X)} \leq \tau_{p,X}(b-a)^{d(\frac{1}{p}-\frac{1}{2})}\norm{f_n}_{\Ellp(\Rd;X)}
\end{align}
for each $n\in\N$. Since the same estimate holds with $f_{n}$ replaced by $f_n - f_m$ for $m\in\N$ with $m\geq n$, it follows that $(f_n)_{n\in\N}$ is a Cauchy sequence in $\gamma(\Rd;X)$. Therefore, with $J_{f_{n}}$ as in \eqref{definition integration operator} for each $n\in\N$, $(J_{f_{n}})_{n\in\N}$ converges to some operator $T\in\gamma(\Ell^{2}(\Rd), X)$ as $n\to\infty$. We claim that $T=J_{f}$. Indeed, fix $x^*\in X^*$. Then $(J_{f_n})^{*}x^{*}\to T^*x^*$ in $\Ell^2(\R^d)$. It is straightforward to check that $(J_{f_{n}})^{*}x^{*}=x^{*}\circ f_{n}$ for each $n\in\N$. Moreover, also $(J_{f_{n}})^{*}x^{*}=x^{*}\circ f_{n}\to x^{*}\circ f=(J_{f})^{*}x^{*}$ in $\Ellp(\Rd)$ as $n\to\infty$. Choosing appropriate almost everywhere convergent subsequences we find that $(J_{f})^{*}x^{*}= T^* x^*$, which yields the claim. The required estimate now follows by letting $n\to \infty$ in \eqref{eq:estndep}.

\eqref{item 2} is proved in the same manner.
\end{proof}

The main result of this section is a consequence of the following proposition, which is of independent interest.

\begin{proposition}\label{compact Fourier support multipliers}
Let $X$ be a Banach space with type $p\in[1,2]$ and $Y$ a Banach space with cotype $q\in[2,\infty]$, and let $r\in[1,\infty]$ be such that $\tfrac{1}{r}=\tfrac{1}{p}-\tfrac{1}{q}$. Let $a,b\in\R$ with $a<b$ and let $m:\Rd\to\La(X,Y)$ be an $X$-strongly measurable map such that $\{m(\xi)\!\mid\! \xi\in[a,b]^{d}\}\subseteq\La(X,Y)$ is $\gamma$-bounded. Then there exists a unique bounded operator $T\in\La(\Ell^{p}_{[a,b]^{d}}(\Rd;X),\Ell^{q}_{[a,b]^{d}}(\Rd;Y))$ such that
\begin{align*}
T(f)=\F^{-1}(m\cdot\widehat{f})
\end{align*}
for each $f\in \Sw_{[a,b]^{d}}(\Rd;X)$. Moreover,
\begin{align*}
\|T(f)\|_{\Ellq(\Rd;Y)}\leq \tau_{p,X}c_{q,Y}(b-a)^{d/r}\gamma(\{m(\xi)\mid \xi\in[a,b]^{d}\})\,\|f\|_{\Ellp(\Rd;X)}
\end{align*}
for all $f\in\Ell^{p}_{[a,b]^{d}}(\Rd;X)$.
\end{proposition}

It follows from an example in \cite{Rozendaal-Veraar16Fourier} that in certain cases the $\gamma$-boundedness condition is necessary in Proposition \ref{compact Fourier support multipliers} and in the results that follow from it.

\begin{proof}
By Lemma \ref{lem:approxsmooth} it suffices to show that $\F^{-1}(m\cdot\widehat{f}\,)\in\Ellq(\Rd;Y)$ with
\begin{align*}
\|\F^{-1}(m\cdot\widehat{f}\,)\|_{\Ellq(\Rd;Y)}
\leq \tau_{p,X}c_{q,Y}(b-a)^{d/r}\gamma(\{m(\xi)\mid \xi\in[a,b]^{d}\})\,\|f\|_{\Ellp(\Rd;X)}
\end{align*}
for all $f\in\Sw_{[a,b]^{d}}(\Rd;X)$. To this end, fix $f\in\Sw_{[a,b]^{d}}(\Rd;X)$ and first assume that $[\xi\mapsto m(\xi)x]\in\Ce^{\infty}_{\text{c}}\!(\Rd;Y)$ for all $x\in X$. Then in fact $m(\cdot)x\in\gamma(\Rd;Y)$, by \eqref{embedding Schwartz functions and gamma-functions}, and by Lemma \ref{ideal property gamma} also $\ind_{[a,b]^{d}}(\cdot)m(\cdot)x\in\gamma(\Rd;Y)$ for all $x\in X$. Now use Lemma \ref{compactly supported functions}, \eqref{Fourier transform on gamma}, Theorem \ref{gamma-multiplier theorem}, \eqref{Fourier transform on gamma} and Lemma \ref{compactly supported functions} in sequence to obtain that $\F^{-1}(m\cdot \widehat{f}\,)\in\Ellq(\Rd;Y)$ with
\begin{align*}
\|\F^{-1}(m\cdot \widehat{f}\,)\|_{\Ellq(\Rd;Y)}&\leq c_{q,Y}(b-a)^{d(\frac{1}{2}-\frac{1}{q})}\|\F^{-1}(m\cdot \widehat{f}\,)\|_{\gamma(\Rd;Y)}\\
&=c_{q,Y}(b-a)^{d(\frac{1}{2}-\frac{1}{q})}\|m\cdot \widehat{f}\|_{\gamma(\Rd;Y)}\\
&= c_{q,Y}(b-a)^{d(\frac{1}{2}-\frac{1}{q})}\|\ind_{[a,b]^{d}}m\cdot \widehat{f}\|_{\gamma(\Rd;Y)}\\
&\leq c_{q,Y}(b-a)^{d(\frac{1}{2}-\frac{1}{q})}\gamma(\{m(\xi)\mid \xi\in[a,b]^{d}\})\|\widehat{f}\|_{\gamma(\Rd;X)}\\
&=c_{q,Y}(b-a)^{d(\frac{1}{2}-\frac{1}{q})}\gamma(\{m(\xi)\mid \xi\in[a,b]^{d}\})\|f\|_{\gamma(\Rd;X)}\\
&\leq \tau_{p,X}c_{q,Y}(b-a)^{d(\frac{1}{p}-\frac{1}{q})}\gamma(\{m(\xi)\mid \xi\in[a,b]^{d}\})\|f\|_{\Ellp(\Rd;X)},
\end{align*}
as required.

Now let $m:\Rd\to\La(X,Y)$ be a general $X$-strongly measurable map such that $\{m(\xi)\mid \xi\in[a,b]^{d}\})\subseteq\La(X,Y)$ is $\gamma$-bounded. Since $\F^{-1}(m\cdot \widehat{f}\,)=\F^{-1}(\ind_{[a,b]^{d}}m\cdot \widehat{f}\,)$, we may assume that $\supp(m)\subseteq[a,b]^{d}$. Let $(h_{n})_{n\in\N}\subseteq\Ce_{\text{c}}^{\infty}\!(\Rd)$ be an approximate identity with $\norm{h_{n}}_{\Ell^{1}\!(\Rd)}\leq 1$ for all $n\in\N$, and define $m_{n}(\xi)x:=(h_{n}\ast m(\cdot)x)(\xi)$ for $n\in\N$, $x\in X$ and $\xi\in\Rd$. Then $[\xi\mapsto m_{n}(\xi)x]\in\Ce^{\infty}_{\text{c}}\!(\Rd;Y)$ and $m(\xi)x=\lim_{n\to\infty}m_{n}(\xi)x$ for all $x\in X$ and almost all $\xi\in\Rd$. Moreover,
\begin{align*}
\gamma(\{m_{n}(\xi)\mid \xi\in\Rd\})\leq \gamma(\{m(\xi)\mid \xi\in\Rd\})
\end{align*}
for each $n\in\N$, by \eqref{Kahane contraction principle} (see also \cite[Corollary 2.14]{Kunstmann-Weis04}). In particular,
\begin{align*}
\sup_{\xi\in\Rd}\sup_{n\in\N}\|m_{n}(\xi)\|_{\La(X,Y)}<\infty.
\end{align*}
Now, by what we have already shown and by \cite[Lemma 3.1]{Rozendaal-Veraar16Fourier}, $T_{m}(f)\in\Ellq(\Rd;Y)$ with
\begin{align*}
\|T_{m}(f)\|_{\Ellq(\Rd;Y)}&\leq \liminf_{n\to\infty}\|T_{m_{n}}(f)\|_{\Ellq(\Rd;Y)}\\
&\leq \tau_{p,X}c_{q,Y}(b-a)^{d/r}\gamma(\{m(\xi)\mid\xi\in[a,b]^{d}\})\|f\|_{\Ellp(\Rd;X)},
\end{align*}
which concludes the proof.
\end{proof}

\subsection{Multipliers on Besov spaces under type and cotype assumptions}
If $m:\Rd\to\La(X,Y)$ is an $X$-strongly measurable map of moderate growth at infinity such that $\{m(\xi)\mid \xi\in I_{n}\}\subseteq\La(X,Y)$ is $\gamma$-bounded for each $n\in\N_{0}$, then by Proposition \ref{compact Fourier support multipliers} (applied to $\ind_{I_{n}}m$) $m$ is an $(\Ell^{p}_{I_{n}}(\Rd;X),\Ell^{q}_{I_{n}}(\Rd;Y))$-Fourier multiplier for each $n\in\N_{0}$. As in Section \ref{subsec:Fourier multipliers on Besov spaces}, let $T_{m}^{(n)}\in\La(\Ell^{p}_{I_{n}}(\Rd;X),\Ell^{q}_{I_{n}}(\Rd;Y))$ be the unique bounded extension of $T_{m}\!\restriction_{\Sw_{I_{n}}(\Rd;X)}$ to $\Ell^{p}_{I_{n}}(\Rd;X)$. Recall that in \eqref{definition extension} we defined, for $s\in\R$ and $v\in[1,\infty]$,
\begin{align}\label{definition extension (co)type}
\widetilde{T_{m}}(f):=\sum_{n=0}^{\infty}T_{m}^{(n)}(\ph_{n}\ast f)\qquad(f\in\Be^{s}_{p,v}(\Rd;X))
\end{align}
as a limit in $\Sw'(\Rd;Y)$. The following result was already stated in the Introduction as Theorem \ref{main result Besov multipliers typeintro}.

\begin{theorem}\label{main result Besov multipliers type}
Let $X$ be a Banach space with type $p\in[1,2]$ and $Y$ a Banach space with cotype $q\in[2,\infty]$, and let $r\in[1,\infty]$ be such that $\tfrac{1}{r}=\tfrac{1}{p}-\tfrac{1}{q}$. Let $m:\Rd\to\La(X,Y)$ be an $X$-strongly measurable map such that $\left(2^{k\sigma}\gamma(\{m(\xi)\!\mid\! \xi\in I_{k}\})\right)_{k\in\N_{0}}\in\ell^{u}$ for some $\sigma\in\R$ and $u\in[1,\infty]$. Then the operator $\widetilde{T_{m}}$ defined by \eqref{definition extension (co)type} extends $T_{m}$ to a bounded linear map $\widetilde{T_{m}}:\Be^{s}_{p,v}(\Rd;X) \to \Be^{s+\sigma-d/r}_{q,w}(\Rd;Y)$
with
\begin{align*}
\|\widetilde{T_{m}}\|_{\La(\Be^{s}_{p,v}(\Rd;X),\Be^{s+\sigma-d/r}_{q,w}(\Rd;Y))}\!\leq 4^{d/r} \tau_{p,X}c_{q,Y}\norm{\Big(2^{k\sigma}\gamma(\{m(\xi)\mid \xi\in I_{k}\})\Big)_{k}}_{\ell^{u}}
\end{align*}
for all $s\in\R$ and all $v,w\in[1,\infty]$ with $\frac{1}{w}\leq \frac{1}{u}+\frac{1}{v}$.
\end{theorem}
In the case of scalar-valued multipliers the $\gamma$-bound of course reduces to a uniform bound.
For the uniqueness of the extensions we refer to Remark \ref{rem:unique}.
\begin{proof}
First note that $m$ is of moderate growth at infinity, so $T_{m}:\Sw(\Rd;X)\to \Sw'(\Rd;Y)$ is well-defined. By Proposition \ref{compact Fourier support multipliers} applied to $\ind_{I_{k}}m$,
\begin{align*}
\|T_{m}(f)\|_{\Ellq(\Rd;Y)}\leq \tau_{p,X}c_{q,Y}(2\cdot 2^{k+1})^{d/r}\gamma(\{m(\xi)\mid \xi\in I_{k}\})\|f\|_{\Ellp(\Rd;X)}
\end{align*}
for all $k\in\N_{0}$ and all $f\in\Sw_{I_{k}}(\Rd;X)$. Letting $\beta:=\tfrac{d}{r}-\sigma$ and, for $k\in\N_{0}$, $c_{k}:=\tau_{p,X}c_{q,Y}4^{d/r}2^{k\sigma}\gamma(\{m(\xi)\mid \xi\in I_{k}\})$, the proof is concluded by appealing to Proposition \ref{prop:multipliers_on_Besov_spaces}.
\end{proof}

\begin{remark}
It also follows from \cite{Rozendaal15} that the smoothness parameter $\tfrac{d}{r}$ in Theorem \ref{main result Besov multipliers type} is sharp, since the results in \cite{Rozendaal15} are derived from Theorem \ref{main result Besov multipliers type} and are sharp with respect to this parameter (see \cite[Remark 6.5]{Rozendaal15}).
\end{remark}

In the same manner we derive the following result from Proposition \ref{prop:multipliers_on_homogeneous_Besov_spaces}.

\begin{theorem}\label{main result homogeneous Besov multipliers type}
Let $X$ be a Banach space with type $p\in[1,2]$ and $Y$ a Banach space with cotype $q\in[2,\infty]$, and let $r\in[1,\infty]$ be such that $\tfrac{1}{r}=\tfrac{1}{p}-\tfrac{1}{q}$. Let $m:\Rd\to\La(X,Y)$ be an $X$-strongly measurable map such that $\left(2^{k\sigma}\gamma(\{m(\xi)\mid \xi\in J_{k}\})\right)_{k\in\Z}\!\in\ell^{u}(\Z)$ for some $\sigma\in\R$ and $u\in[1,\infty]$. Then \eqref{definition homogeneous extension} defines an extension $\overline{T_{m}}\in \La(\dot{\Be}^{s}_{p,v}(\Rd;X),\dot{\Be}^{s+\sigma-d/r}_{q,w}(\Rd;Y))$ of $\dot{T}_{m}$ such that
\begin{align*}
\|\overline{T_{m}}\|_{\La(\dot{\Be}^{s}_{p,v}(\Rd;X),\dot{\Be}^{s+\sigma-d/r}_{q,w}(\Rd;Y))}\!\leq 4^{d/r} \tau_{p,X}c_{q,Y}\norm{\Big(2^{k\sigma}\gamma(\{m(\xi)\mid \xi\in J_{k}\})\Big)_{k}}_{\ell^{u}(\Z)}
\end{align*}
for all $s\in\R$ and all $v,w\in[1,\infty]$ with $\frac{1}{w}\leq \frac{1}{u}+\frac{1}{v}$.
\end{theorem}

As a consequence we can derive an $(\Ell^p,\Ell^q)$-multiplier result.
\begin{theorem}\label{thm:Lp-Lq multipliers type}
Let $X$ be a Banach space with type $p\in[1,2]$ and $Y$ a Banach space with cotype $q\in[2,\infty]$, and let $r\in[1,\infty]$ be such that $\tfrac{1}{r}=\tfrac{1}{p}-\tfrac{1}{q}$. Let $m:\Rd\setminus\{0\}\to\La(X,Y)$ be an $X$-strongly measurable map such that
\begin{equation}\label{eq:condmultRbdd}
\left(2^{kd/r}\gamma(\{m(\xi)\mid \xi\in J_{k}\})\right)_{k\in\Z}\!\in\ell^{1}(\Z).
\end{equation}
Then $T_{m}$ extends uniquely to a bounded map $\widetilde{T_{m}}\in\La(\Ellp(\Rd;X),\Ellq(\Rd;Y))$ with
\begin{align*}
\|\widetilde{T_{m}}\|_{\La(\Ellp(\Rd;X),\Ellq(\Rd;Y))}\leq C 4^{d/r}\tau_{p,X}c_{q,Y}\|(2^{kd/r}\gamma(\{m(\xi)\!\mid\!\xi\in J_{k}\}))_{k}\|_{\ell^{1}(\Z)},
\end{align*}
where $C\geq0$ is a constant which depends only on $p$ and $d$.
\end{theorem}
In \cite[Theorem 1.1]{Rozendaal-Veraar16Fourier} we derived a similar result, with \eqref{eq:condmultRbdd} replaced by the assumption that $\{|\xi|^{d/r} m(\xi)\mid \xi\in \R^d\setminus\{0\}\}$ is $\gamma$-bounded. However, there we use that $X$ has type $p_{0}>p$ and $Y$ cotype $q_{0}<q$. Note that this is not needed in Theorem \ref{thm:Lp-Lq multipliers type}, at the cost of a more restrictive $\gamma$-boundedness condition.

\begin{proof}
First note that $T_{m}:\dot{\Sw}(\Rd;X)\to\Sw'(\Rd;Y)$ is well-defined since $m$ is of moderate growth at infinity, where we use that $\left(2^{kd/r}\gamma(\{m(\xi)\mid \xi\in J_{k}\})\right)_{k\in\Z}\!\in\ell^{1}(\Z)$. Moreover, since $\dot{\Sw}(\Rd;X)\subseteq\Ellp(\Rd;X)$ is dense, it suffices to show that
\begin{align*}
\|T_{m}(f)\|_{L^q(\R^d;Y)}\leq C4^{d/r}\tau_{p,X}c_{q,Y}\!\|(2^{kd/r}\gamma(\{m(\xi)\mid \xi\in J_{k}\}))_{k}\|_{\ell^{1}(\Z)}\|f\|_{\Ellp(\Rd;X)}
\end{align*}
for all $f\in\dot{\Sw}(\Rd;X)$. To this end, note that it straightforward to show that the contractive inclusion $\dot{\Be}^{0}_{q,1}(\Rd;Y)\hookrightarrow \Ellq(\Rd;Y)$ and the continuous inclusion $\Ellp(\Rd;X)\hookrightarrow \dot{\Be}^{0}_{p,\infty}(\Rd;X)$ hold. Using these inclusions, the required estimate follows from Theorem \ref{main result homogeneous Besov multipliers type} with $u = w=1$ and $v = \infty$.
\end{proof}

Theorem \ref{thm:Lp-Lq multipliers type} can be improved for UMD spaces. For details on UMD spaces we refer to \cite{Bu3, RF86} and to the recent monograph \cite{HNVW1}.
\begin{remark}\label{better with UMD}
Suppose, in addition to the assumptions of Theorem \ref{thm:Lp-Lq multipliers type}, that $X$ is a UMD space with cotype $q_{0}\in[2,\infty)$ and $Y$ is a UMD space with type $p_0\in(1,2]$. Then the homogeneous version of \cite[Proposition 3.1]{Veraar13} (proved in the same way as in the inhomogeneous case) yields the continuous embeddings $\Ellp(\Rd;X)\hookrightarrow \dot{\Be}^{0}_{p,q_{0}}(\Rd;X)$ and $\dot{\Be}^{0}_{q,p_0}(\Rd;X)\hookrightarrow \Ell^{q}(\Rd;X)$. Following the proof of Theorem \ref{thm:Lp-Lq multipliers type} it then suffices to assume that for $\frac{1}{r_0} = \frac{1}{p_0} - \frac{1}{q_0}$, one has
\[\left(2^{kd/r}\gamma(\{m(\xi)\mid \xi\in I_{k}\})\right)_{k\in\Z}\!\in\ell^{r_0}.\]
\end{remark}

\begin{remark}\label{rem: weighted gamma-boundedness}
It is straightforward to check that the condition on $m$ in Theorem \ref{thm:Lp-Lq multipliers type} holds if $\{\abs{\xi}^{\sigma}m(\xi)\mid \abs{\xi}\geq 1\}\cup \{\abs{\xi}^{\mu}m(\xi)\mid \abs{\xi}\leq 1\}\subseteq\La(X,Y)$ is $\gamma$-bounded for some
$\sigma,\mu\in\R$ with $\sigma>\tfrac{d}{r}>\mu$. The exponent $\tfrac{d}{r}$ cannot be improved, as can be seen from the scalar case and the Hardy--Littlewood--Sobolev inequality (see \cite[Theorem 6.1.3]{Grafakos09}).

As a consequence, by applying the same method as in \cite[Lemma 3.26 and Proposition 3.27]{Rozendaal-Veraar16Fourier}, one sees that in certain cases the type $p$ of $X$ and cotype $q$ of $Y$ are necessary for Theorem \ref{thm:Lp-Lq multipliers type} and hence for Theorem \ref{main result homogeneous Besov multipliers type}.
\end{remark}

\section{Extrapolation}\label{extrapolation}

In this section we give a proof of \cite[Theorem 4.1]{Rozendaal-Veraar16Fourier} and use it to extrapolate our Fourier multiplier results on Besov spaces to different integrability exponents.  In order to prove the extrapolation result we extend several results of H\"ormander in \cite{Hormander60} to the vector-valued setting and $p\leq q$.

For $a,p,q\in [1, \infty]$ with $a\neq \infty$ consider the following identity:
\begin{align}\label{eq:admissble}
\frac{1}{p} - \frac{1}{q} = 1-\frac{1}{a}.
\end{align}

\subsection{Kernels and extrapolation}\label{subsec:ExtraKernel}

Throughout this section we fix $d\in\N$ and Banach spaces $X$ and $Y$. Consider the following variant of H\"ormander's condition which we formulate in the strong operator topology:
\begin{enumerate}
\item[(H)$_a$]
Let $K:\Rd\setminus\{0\}\to \La(X,Y)$ be such that for all $x\in X$, $t\mapsto K(t) x$  is locally integrable on $\Rd\setminus\{0\}$. Suppose there exists a constant $C_{H,a}\geq 0$ such that
\begin{align*}
\Big(\int_{|s|\geq 2|t|} \|K(s-t)x - K(s)x\|^{a} \, ds\Big)^{\frac1a}\leq C_{H,a} \|x\|\quad (x\in X, t\in \R^d\setminus\{0\}).
\end{align*}
We denote the infimum over all {\em H\"ormander constants} $C_{H,a}$ by $C_{H,a}(K)$.
\end{enumerate}

\begin{remark}
In particular, the condition (H)$_a$ holds with constant $C_{H,a}>0$ if $K$ is $X$-strongly measurable and
\begin{align}\label{eq:Hormanderuniform}
\Big(\int_{|s|\geq 2|t|} \|K(s-t) - K(s)\|_{\La(X,Y)}^{a} \, ds\Big)^{\frac1a}\leq C_{H,a}\quad(t\in \R^d\setminus \{0\}),
\end{align}
where we assume that the integrand is measurable (or at least $\one_{|s|\geq 2|t|} \|K(s-t) - K(s)\|\leq f_t(s)$, where $f$ is measurable and satisfies $\|f_t\|_{\Ell^a(\R^d)}\leq C_{H,a}$ for all $t\neq 0$).
Under the appropriate measurability conditions on $K^*$, \eqref{eq:Hormanderuniform} implies (H)$_a$ for $K^*$ as well.

The advantage of (H)$_a$ over \eqref{eq:Hormanderuniform} will become clear in the proof of Theorem \ref{thm:extrapolmult}.
\end{remark}

Let $\Ellinftyc(\Rd)\subseteq \Ellinfty(\Rd)$ denote the subset of functions which have compact support. Let $\Ellinftyc(\Rd)\otimes X$ be the linear span of the functions $t\mapsto (f\otimes x)(t) := f(t) x$ where $f\in \Ellinftyc(\Rd)$ and $x\in X$.

Let $K:\Rd\setminus\{0\}\to \La(X,Y)$ be such that for all $x\in X$, $t\mapsto K(t) x$  is locally integrable on $\Rd\setminus\{0\}$. For a bounded linear operator $T:\Ellp(\Rd;X)\to \Ellq(\Rd;Y)$ and $p,q\in [1, \infty]$, consider the following condition: for all $f\in \Ellinftyc(\Rd)\otimes X$
\begin{equation}\label{eq:TreprK}
T f(s) = \int_{\R^d} K(s-t) f(t)\,  \ud t\quad \text{for almost all }s\in (\supp(f))^c.
\end{equation}

The following result is a vector-valued extension of \cite[Theorem 2.2]{Hormander60}. Recall that the norm of the space $\Ell^{p,\infty}(\Rd;X)$ is given by
\begin{align}\label{eq:weakLpnotation}
\|f\|_{\Ell^{p,\infty}(\R^d;X)} := \sup_{\alpha>0} \alpha \lambda_f(\alpha)^{\frac1p}<\infty,
\end{align}
where $\lambda_f(\alpha) := \mu(\{s\in \R^d\mid\|f(s)\|_{X}>\alpha\})$ for $\alpha>0$ and $\mu$ is the Lebesgue measure.

\begin{proposition}[Extrapolation to $\Ell^1\to \Ell^{a,\infty}$]\label{prop:weakLa}
Let $p,q\in(1,\infty]$ and $a\in[1,\infty)$ satisfy \eqref{eq:admissble}. Let $K:\Rd\setminus\{0\}\to\La(X,Y)$ satisfy condition (H)$_a$. Let $T:\Ellp(\Rd;X)\to \Ellq(\Rd;Y)$ be a bounded linear operator of norm $B$ satisfying \eqref{eq:TreprK}. Then $T:\Ell^1(\R^d;X)\to \Ell^{a,\infty}(\R^d;Y)$ is bounded and
\begin{align*}
\|T\|_{\La(\Ellone(\Rd;X),\Ell^{a,\infty}(\Rd;Y))} \leq C_{d,a} B + 4 C_{H,a}(K),
\end{align*}
where $C_{d,a}:=2+2 d^{\frac{d}{2a}} 4^{\frac{d}{a}}$.
\end{proposition}

\begin{proof}
We adopt the presentation from \cite[Theorem 4.3.3]{Grafakos08} and show that it extends to the vector-valued setting with general $p\leq q$.

Assume $q<\infty$ (the case $q=\infty$ is left to the reader, see \cite[Exercise 4.3.7]{Grafakos08}).
In order to prove the result it suffices to show that, for each simple function $f\in\Ellinftyc(\Rd;X)$ of norm $\|f\|_{L^{1}(\Rd;X)}\leq 1$,
\begin{align}\label{eq:toshowTbddweak}
\|Tf\|_{a,\infty}\leq C_{d,a} B + 4 C_{H,a}(K).
\end{align}

We apply the Calder\'on-Zygmund decomposition of height $\gamma \alpha^a$ to write $f$ as the sum of a good and bad part $g+b$. Here $\alpha>0$ is fixed for the moment, and we set
\begin{equation}\label{eq:choicegamma}
\gamma := B^{-a}2^{-(d+a)}.
\end{equation}
To obtain this decomposition note that \cite[Theorem 4.3.1 and Remark 4.3.2]{Grafakos08} have a straightforward generalization to the vector-valued setting.
The decomposition given there yields $f = g+b$, where $b  = \sum_{j\in\N} b_j$ for simple $b_j\in\Ellinftyc(\Rd)\otimes X$,
and the existence of a sequence of disjoint cubes $(Q_j)_{j\in\N}$ in $\Rd$ such that the following properties are satisfied, for each $j\in\N$:
\begin{align}\label{eq:CZprop1}
\|g\|_{1}  \leq 1,\quad & \|g\|_{\infty}\leq 2^d \gamma \alpha^a,\quad\|g\|_{p}\leq 2^{\frac{d}{p'}} \gamma^{\frac{1}{p'}}\alpha^{\frac{a}{p'}}, \quad\|b\|_1\leq 2,
\\ \label{eq:CZprop2}
\supp(b_j) & \subseteq Q_j,\quad\int b_j(t) \,\ud t= 0,\quad\sum_{k\in\N} |Q_k| \leq \frac{1}{\gamma\alpha^a}.
\end{align}

By subadditivity we can write (see below \eqref{eq:weakLpnotation} for the definition of $\lambda_f$)
\begin{align*}
\lambda_{Tf}(\alpha) \leq \lambda_{Tg}(\alpha/2) + \lambda_{Tb}(\alpha/2).
\end{align*}
The good part $Tg$ can be estimated directly using the boundedness of $T$ and \eqref{eq:CZprop1}:
\begin{align}\label{eq:goodpart}
\lambda_{Tg}(\alpha/2) \leq \frac{2^q}{\alpha^q} \|Tg\|_q^q \leq \frac{2^q B^q}{\alpha^q} \|g\|_p^q
\leq  \frac{2^q B^q}{\alpha^q}  2^{\frac{q d}{p'}}  \alpha^{\frac{q a}{p'}} \gamma^{\frac{q}{p'}}= \frac{2^a B^a}{\alpha^a} ,
\end{align}
where we used \eqref{eq:admissble} and the choice of $\gamma$ given in \eqref{eq:choicegamma} in the last step.
The bad part we split into two parts again.
Let $Q_j^*$ be the unique cube with sides parallel to $Q_j$ and the same center as $Q_j$ but such that $\ell(Q_j^*) = 2\sqrt{d} \ell(Q_j)$, where $\ell(Q)$ denotes the side length of a cube $Q$. Setting $\Omega = \bigcup_{j\geq 1}Q_j^*$, we can write
\begin{align}\label{eq:badpartsplit}
\lambda_{Tb}(\alpha/2) & \leq |\Omega| + \big|\{s\in\Omega^c: \|Tb(s)\|_Y>\frac{\alpha}{2}\}\big|
 \leq  |\Omega| + \frac{2^a}{\alpha^a} \int_{\Omega^c} \|Tb(s)\|^a \,\ud s
\end{align}
By the choice of $Q_j^*$ and by \eqref{eq:CZprop2} and \eqref{eq:choicegamma},
\begin{align}\label{eq:badpart1}
|\Omega| \leq \sum_{j\geq 1} |Q_j^*| \leq \frac{(2\sqrt{d})^d }{\gamma\alpha^a} = \frac{4^d d^{\frac{d}{2}} 2^{a} B^a} {\alpha^a}.
\end{align}
For the second part of \eqref{eq:badpartsplit} note that by the triangle inequality
\begin{align}\label{eq:badpart2}
\|\one_{\Omega^c} Tb\|_{a}\leq \sum_{j\geq 1}\|\one_{\Omega^c} Tb_j\|_{a}\leq \sum_{j\geq 1}\|\one_{(Q_j^*)^c} Tb_j\|_{a}.
\end{align}
We first estimate each term of this series separately. Let $t_j$ denote the center of $Q_j$. By \eqref{eq:TreprK} and \eqref{eq:CZprop2} (twice),
\begin{align*}
\|\one_{(Q_j^*)^c} Tb_j\|_{a} &= \Big(\int_{\R^d\setminus Q_j^*} \Big\| \int_{Q_j} K(s-t)b_j(t) \,\ud t\Big\|^a \,\ud s\Big)^{\frac1a}
\\ & = \Big(\int_{(Q_j^*)^c} \Big\| \int_{Q_j} K(s-t)b_j(t) - K(s-t_j) b_j(t) \,\ud t\Big\|^a \,\ud s\Big)^{\frac1a}
\\ & \stackrel{(i)}{\leq} \int_{Q_j} \Big( \int_{(Q_j^*)^c} \|K(s-t)b_j(t) - K(s-t_j) b_j(t)\|^a \,\ud s\Big)^{\frac1a} \,\ud t
\\ & = \int_{Q_j} \Big( \int_{(Q_j^*)^c - t_j} \|K(s-(t-t_j))b_j(t) - K(s) b_j(t)\|^a \,\ud s\Big)^{\frac1a} \,\ud t
\\ & \stackrel{(ii)}{\leq} \int_{Q_j} \Big( \int_{|s|\geq 2|t-t_j|} \|K(s-(t-t_j))b_j(t) - K(s) b_j(t)\|^a \,\ud s\Big)^{\frac1a} \,\ud t
\\ & \stackrel{(iii)}{\leq} C_{H,a}(K) \int_{Q_j} \|b_j(t)\| \,\ud t = C_{H,a}(K) \|b_j\|_1.
\end{align*}
In (i) we applied Minkowski's inequality. The estimate (ii) follows from $|s|\geq \frac12 \ell(Q_j^*) = \sqrt{d} \ell(Q_j) \geq 2|t-t_j|$ for $s\in (Q_j^*)^{\mathrm{c}}-t_{j}$ and $t\in Q_j$. In (iii) we applied (H)$_a$. With \eqref{eq:badpart2}, \eqref{eq:CZprop1}  and \eqref{eq:CZprop2} we can conclude that
\begin{align}\label{eq:badpart2concl}
\|\one_{\Omega^c} Tb\|_{a}\leq C_{H,a}(K) \sum_{j\geq 1}\|b_j\|_1 = C_{H,a}(K)  \|b\|_{1} \leq 2 C_{H,a}(K).
\end{align}
Now \eqref{eq:badpartsplit}, \eqref{eq:badpart1} and \eqref{eq:badpart2concl} yield
\begin{align}\label{eq:badpartfinal}
\lambda_{Tb}(\alpha/2) \leq \frac{4^d d^{\frac{d}{2}} 2^{a} B^a}{\alpha^a} +  \frac{4^a C_{H,a}(K)^a}{\alpha^a}.
\end{align}

Finally, combining the good part \eqref{eq:goodpart} and the bad part \eqref{eq:badpartfinal}, we obtain
\begin{align*}
\alpha  \lambda_{Tf}(\alpha)^{\frac1a}  \leq  \Big(2^a B^a +  4^d d^{\frac{d}{2}} 2^{a} B^a +  4^a C_{H,a}(K)^a\Big)^{\frac1a}.
\end{align*}
Now \eqref{eq:toshowTbddweak} follows from the estimate $(x^{\alpha} + y^{\alpha})^{\frac1{\alpha}} \leq x+y$, for $x,y>0$, and by taking the supremum over all $\alpha>0$.
\end{proof}

\begin{corollary}[Extrapolation I, kernel condition]\label{cor:Hormander}
Let $p_0,q_0\in (1, \infty]$ and $a\in[1,\infty)$ be such that $\frac{1}{p_0} - \frac{1}{q_0} = 1-\frac{1}{a}$.
Let $K:\Rd\setminus\{0\}\to\La(X,Y)$ satisfy (H)$_a$. Let $T:\Ell^{p_0}(\Rd;X)\to \Ell^{q_0}(\Rd;Y)$ be a bounded linear operator of norm $B$ satisfying \eqref{eq:TreprK}.
Then, for all $(p,q)$ satisfying $p\in (1, p_0]$ and \eqref{eq:admissble}, $T:\Ellp(\Rd;X)\to \Ellq(\Rd;Y)$ is bounded and
\begin{align*}
\|T\|_{\La(\Ellp(\Rd;X),\Ellq(\Rd;Y))} \leq C_{p_0, q_0, p, d} (B + C_{H,a}(K)),
\end{align*}
where $C_{p_0, q_0, p, d}\sim (p-1)^{-1}$ as $p\downarrow 1$.
\end{corollary}
\begin{proof}
By Proposition \ref{prop:weakLa} we find that also $T:\Ell^1(\R^d;X)\to \Ell^{a,\infty}(\R^d;Y)$ is bounded. From the Marcinkiewicz interpolation theorem (see \cite{HuntWeiss} for a formulation with explicit constants which extends to the vector valued setting), we deduce the required boundedness and estimate.
\end{proof}

Under other conditions on $K$ we can extrapolate to $p>p_0$:

\begin{corollary}[Extrapolation II, kernel condition]\label{cor:HormanderII}
Let $p_0,q_0,a\in [1, \infty)$ with $q_0\neq 1$
be such that $\frac{1}{p_0} - \frac{1}{q_0} = 1-\frac{1}{a}$. Let $K:\Rd\to\La(X,Y)$ be such that $K(\cdot)x$ and $K^{*}(\cdot)y$ are locally integrable on $\Rd$ for all $x\in X$ and $y^{*}\in Y^{*}$. Suppose that $K^{*}$ satisfies (H)$_a$. Let $T:\Ell^{p_0}(\Rd;X)\to \Ell^{q_0}(\Rd;Y)$ be a bounded linear operator of norm $B$ such that $T f = K*f$ for all $f\in \Ellinftyc(\Rd)\otimes X$.
Then, for all $(p,q)$ satisfying $q\in [q_0,\infty)$ and \eqref{eq:admissble}, $T:\Ellp(\Rd;X)\to \Ellq(\Rd;Y)$ is bounded and
\begin{align*}
\|T\|_{\La(\Ellp(\Rd;X),\Ellq(\Rd;Y))} \leq C_{p_0, q_0, q, d} (B+ C_{H,a}(K^*)),
\end{align*}
where $C_{p_0, q_0, p, d}\sim q$ as $q\uparrow \infty$.
\end{corollary}

\begin{proof}
For $g\in \Ellinftyc(\Rd)\otimes Y^{*}$ and $t\in \Rd$, let $Sg(t)\in X^{*}$ be defined by
\begin{align*}
S g(t) := \int_{\Rd}  K(s-t)^* g(s)\,\ud s.
\end{align*}
One can check that
\begin{equation}\label{eq:dualityTS}
  \lb f, Sg \rb = \lb Tf, g\rb\quad (f\in \Ellinftyc(\Rd)\otimes X, g\in\Ellinftyc(\Rd)\otimes Y^{*}).
\end{equation}
Thus, by a density argument, $S$ extends to a bounded mapping from $\Ell^{q_0'}(\R^d;Y^*)$ into $\Ell^{p_0'}(\R^d;X^*)$ of norm $B$.
By Corollary \ref{cor:Hormander}, $S$ extends to a bounded mapping from $\Ell^{q'}(\Rd;Y^*)$ into $\Ell^{p'}(\Rd;X^*)$ for all $(p,q)$ satisfying $q'\in (1, q_0']$ and \eqref{eq:admissble}. The proof is concluded by using \eqref{eq:dualityTS} once again.
\end{proof}

\begin{remark}\label{rem:comparisonKernel} \
\begin{enumerate}[(i)]
\item The extrapolation result of H\"ormander \cite[Theorem 2.1]{Hormander60} was extended in \cite[Theorem 2]{BCP62} to the vector-valued setting in the case $p=q$. Moreover, they also considered extensions to $\Ell^p$ for vectors $p\in (1, \infty)^d$.
    It is sometimes overlooked that the kernels in \cite[Theorem 2]{BCP62} are assumed to be locally integrable on $\R^d$, hence in particular integrable at zero. This assumption only plays a role in the duality argument. See \cite[Theorem 3]{BCP62} for other possible conditions in the case where $X$ and $Y$ are Hilbert spaces.

\item To weaken the assumption of local integrability on $\R^d$ in the operator-valued setting in Proposition \ref{prop:weakLa}, Corollary \ref{cor:Hormander} and Corollary \ref{cor:HormanderII}, one can check that it suffices to assume $t\mapsto K(t)x$ and $t\mapsto K(t)y^*$ are locally integrable on $\R^d$ only for $x$ and $y^*$ in a dense subspace of $X$ and $Y^*$ respectively.

\item A slightly different presentation in the case $p=q\in [1, \infty]$ is given in \cite[Theorem V.3.4]{GarciaRubiobook}, where the condition \eqref{eq:Hormanderuniform} is used. The argument given there has the advantage that no duality arguments are required. To extrapolate to $p\in (p_0, \infty)$ one first proves that $T$ maps $\Ellinfty$ into BMO, after which an interpolation argument can be applied again.
\end{enumerate}
\end{remark}

\subsection{Multipliers and extrapolation\label{subsec:ExtraMult}}

In this section we use the extrapolation results from above to prove an extension of the extrapolation theorem for Fourier multipliers in \cite[Theorem 2.5]{Hormander60}.
Let $m:\Rd\setminus\{0\}\to\La(X,Y)$ be a strongly measurable map of moderate growth at zero and infinity.
For $r\in [1, \infty)$, $\varrho\in [1, \infty)$ and $n\in \N$, consider the following variants of the Mihlin--H\"ormander condition:
\begin{enumerate}
\item[(M1)$_{r,\varrho,n}$]
There exists a constant $M_1\geq 0$ such that for all multi-indices $|\alpha|\leq n$,
\begin{align*}
R^{|\alpha| + \frac{d}{r} - \frac{d}{\varrho}} \Big(\int_{R\leq |\xi|<2R} \|\partial^{\alpha} m(\xi)x\|^{\varrho} \, d\xi\Big)^{1/\varrho}\leq M_1\|x\|\qquad (x\in X, R>0).
\end{align*}
\item[(M2)$_{r,\varrho,n}$]
There exists a constant $M_2\geq 0$ such that for all multi-indices $|\alpha|\leq n$
\begin{align*}
R^{|\alpha| + \frac{d}{r} - \frac{d}{\varrho}} \Big(\int_{R\leq |\xi|<2R} \|\partial^{\alpha} m(\xi)^*y^*\|^{\varrho} \, d\xi\Big)^{1/\varrho}\leq M_2\|y^*\|\qquad (y^*\in Y^*,  R>0).
\end{align*}
\end{enumerate}
For $\varrho = 2$, $r=1$, $X = Y = \R$, condition (M1)$_{r,\varrho,n}$ reduces to the classical H\"ormander condition in \cite[Theorem 2.5]{Hormander60} (see also \cite[Theorem 5.2.7]{Grafakos08}).

Now we can now prove the main result of this section. The following theorem was already stated as \cite[Theorem 4.1]{Rozendaal-Veraar16Fourier} without proof, and extends \cite[Theorem 2.5]{Hormander60} to the vector-valued setting and to general exponents $p,q\in (1, \infty)$.

\begin{theorem}[Extrapolation, multiplier condition]\label{thm:extrapolmult}
Let $p_0,q_0,r\in [1, \infty]$ with $r\neq 1$ be such that $\frac{1}{p_0} - \frac{1}{q_0} = \frac{1}{r}$. Let $m:\Rd\setminus\{0\}\to\La(X,Y)$ be a strongly measurable map of moderate growth at zero and infinity. Suppose that $T_m:\Ell^{p_0}(\R^d;X) \to \Ell^{q_0}(\R^d;Y)$ is bounded of norm $B$.
\begin{enumerate}[$(1)$]
\item Suppose that $p_0\in (1, \infty]$, $Y$ has Fourier type $\varrho\in [1, 2]$ with $\varrho\leq r$, and (M1)$_{r, \varrho, n}$ holds for $n := \lfloor \frac{d}{\varrho} - \frac{d}{r} \rfloor+1$.
Then $T_{m}\in\La(\Ellp(\Rd;X),\Ellq(\Rd;Y))$ and
\begin{align}\label{eq:toprovebddextr}
\|T_{m}\|_{\La(\Ellp(\Rd;X),\Ellq(\Rd;Y))} \leq C_{p_{0},q_{0},p,d} (M_1 + B)
\end{align}
for all $(p,q)$ such that $p\in (1, p_0]$ and $\tfrac{1}{p}-\tfrac{1}{q}=\tfrac{1}{r}$, where $C_{p_0, q_0, p, d}\sim (p-1)^{-1}$ as $p\downarrow 1$.
\item Suppose that $q_0\in (1, \infty)$, $X$ has Fourier type $\varrho\in [1, 2]$ with $\varrho\leq r$, and (M2)$_{r, \varrho, n}$ holds for $n := \lfloor \frac{d}{\varrho} - \frac{d}{r} \rfloor+1$. Then $T_{m}\in\La(\Ellp(\Rd;X),\Ellq(\Rd;Y))$ and
\begin{align}\label{eq:toprovebddextr2}
\|T_m\|_{\calL(\Ellp(\R^d;X),\Ellq(\R^d;Y))} \leq C_{p_{0},q_{0},q,d} (M_2 + B),
\end{align}
for all $(p,q)$ satisfying $q\in [q_0,\infty)$ and $\tfrac{1}{p}-\tfrac{1}{q}=\tfrac{1}{r}$, where $C_{p_0, q_0, q, d}\sim q$ as $q\uparrow \infty$.
\end{enumerate}
\end{theorem}

As in \cite{Rozendaal-Veraar16Fourier}, we can deduce the following corollary from Theorem \ref{thm:extrapolmult}:

\begin{corollary}\label{cor:extrapol m uniform}
Let $p_0,q_0,r\in [1, \infty]$ with $q_0\neq1 $ and $r\neq 1$ be such that $\frac{1}{p_0} - \frac{1}{q_0} = \frac{1}{r}$.
Let $X$ and $Y$ both have Fourier type $\varrho\in [1, 2]$ $\varrho\leq r$ and let $n := \lfloor \frac{d}{\varrho} - \frac{d}{r} \rfloor+1$.
Let $m:\Rd\setminus\{0\}\to\La(X,Y)$ be such that, for all multi-indices $|\alpha|\leq n$,
\begin{align}\label{eq:mihlin}
\|\partial^{\alpha} m(\xi)\|\leq C|\xi|^{-|\alpha| - \frac{d}{r}}, \ \ \xi\in \R^d\setminus \{0\}.
\end{align}
Suppose that $T_m:\Ell^{p_0}(\Rd;X) \to \Ell^{q_0}(\Rd;Y)$ is bounded of norm $B$.
Then, for all exponents $p$ and $q$ satisfying $1<p\leq q<\infty$ and $\frac1p-\frac1q  = \frac1r$, $T_m:\Ellp(\R^d;X) \to \Ellq(\R^d;Y)$ is bounded and
\begin{align*}
\|T_m\|_{\La(\Ellp(\Rd;X),\Ellq(\Rd;Y))} \leq C_{p,q,d} (B+C)
\end{align*}
for some constant $C_{p,q,d}\geq 0$.
\end{corollary}

Note that one can always take $\varrho= 1$ and $n= \lfloor \frac{d}{r'}\rfloor+1$ in the results above.

\begin{remark}
If $p_0 = q_0=1$, then the results above are true with $\varrho=1$. Indeed, one can repeat the proof of Theorem \ref{thm:extrapolmult} using the trivial Fourier type $1$ of $\calL(X,Y)$ and apply \cite[Theorem V.3.4]{GarciaRubiobook} (see Remark \ref{rem:comparisonKernel}).
\end{remark}

\begin{proof}[Proof of Theorem \ref{thm:extrapolmult}]
We follow the line of reasoning from \cite[Theorem 5.2.7]{Grafakos08}.

(1): \ By replacing $\varrho$ by a slightly smaller number if necessary we can assume that $\varrho<r$. Let $\eta\in\Sw(\Rd)$ be such that
\begin{align}\label{eq:definition eta}
\wh{\eta}(\xi)\in[0,1]\text{ for }\xi\in\Rd,\quad \wh{\eta}=1\text{ if }\abs{\xi}\leq 1,\quad \wh{\eta}=0\text{ if }\abs{\xi}\geq \tfrac{3}{2}.
\end{align}
Let $(\zeta_{j})_{j\in\Z}\subseteq\Sw(\Rd)$ be such that $\wh{\zeta_{0}}(\xi):=\wh{\eta}(\xi)-\wh{\eta}(2\xi)$ and $\wh{\zeta_{j}}(\xi):=\wh{\zeta_{0}}(2^{-j}\xi)$ for $\xi\in\Rd$ and $j\in\Z$. Then $\supp(\wh{\zeta_{j}})\subseteq\{\xi\in\Rd\mid \abs{\xi}\in[2^{j-1},\tfrac{3}{2}2^{j}]\}$ and
\begin{align*}
\sum_{j\in \Z} \wh{\zeta_{j}}(\xi)=  1\qquad (\xi\in\Rd\setminus\{0\}).
\end{align*}
Set $m_j (\xi) :=\widehat{\zeta_{j}}(\xi) m(\xi)$ for $j\in \Z$. Let $K_j:= \F^{-1} (m_j)\in L^\infty(\R^d;\calL(X,Y))$.
We fix $N\in \N$ and let $K^{(N)} = \sum_{j=-N}^N K_j$. Then ,
\begin{align}\label{eq:identityKm}
K^{(N)}*f = T_{m^{(N)}} f = T_m T_g f,    \ \ \ \text{for all} \  \ f\in C^\infty_c(\R^d)\otimes X,
\end{align}
where $m^{(N)} =\wh{g}^{(N)} m$ and $\wh{g}^{(N)} =  \sum_{j=-N}^N \wh{\zeta_{j}} = (\wh{\eta}(2^{-N-1}\xi)-\wh{\eta}(2^{N+1}\xi))$.
Since $\|g^{(N)}\|_{\Ell^1(\R^d)}\leq 2$, also
$\|K^{N}*f\|_{\Ell^{q_0}(\R^d;Y)}\leq 2 B \|f\|_{\Ell^{p_0}(\R^d;X)}$ for all $f\in C^\infty_c(\R^d)\otimes X$. We will prove that \eqref{eq:toprovebddextr} holds with $m$ replaced by $m^{(N)}$. Since $m^{(N)}(\xi)\to m(\xi)$ for almost all $\xi\in\Rd$ as $N\to\infty$, \cite[Lemma 3.1]{Rozendaal-Veraar16Fourier} would then conclude the proof of (1). For this we will check the conditions of Corollary \ref{cor:Hormander} with constants independent of $N$.

By the preceding discussion, from now on we may assume that $m = m^{(N)}$ and $K = K^{(N)}$. We first claim that \eqref{eq:identityKm} extends to all $f\in \Ell^\infty_c(\R^d)\otimes X$. This is clear if $p_0<\infty$ by a density argument. Next consider $p_0 = q_0 = \infty$. Using the Hahn-Banach theorem we can reduce to the scalar case. Fix $x\in X$ and $y^*\in Y$ and let $K_{x, y^*}(t) := \lb K(t)x, y^*\rb$ and $m_{x, y^*}(t) := \lb mx, y^*\rb$. Since $T_{m_{x, y^*}}$ is bounded on $\Ell^\infty(\R^d)$, by duality $T_{m_{x, y^*}}$ is also bounded on $\Ell^1(\R^d)$. Now we can apply the same density argument as before.

Let $\delta>0$ be a constant which is chosen suitably small below, and let $x\in X$. We claim that there exists a constant $C_{d}\geq0$ such that
\begin{align}
\sup_{j\in \Z} T_{1,j} &:=  \sup_{j\in \Z} \Big(\int_{\R^d} \|K_j(s)x\|^{r'} (1+2^j |s|)^{\delta} \, ds \Big)^{\frac{1}{r'}} \leq C_d M_1\|x\|_{X}, \label{eq:HormKern1}
\\ \sup_{j\in \Z} T_{2,j} &:= \sup_{j\in \Z} 2^{-j} \Big(\int_{\R^d} \|\nabla K_j(s)x\|^{r'} (1+2^j |s|)^{\delta} \, ds \Big)^{\frac{1}{r'}} \leq C_d M_1\|x\|_{X}.\label{eq:HormKern2}
\end{align}
Indeed, by H\"older's inequality, with $1 = \frac{r'}{\varrho'} + \frac{1}{b}$ for some $b\in[1,\infty)$,
\begin{align*}
T_{1,j}&=\Big(\int_{\R^d} \|K_j(s)x\|^{r'} (1+2^j |s|)^{\delta} \, ds \Big)^{\frac{1}{r'}}
\\ & \leq \Big(\int_{\R^d} \|K_j(s)x\|^{\varrho'}  (1+2^j |s|)^{\varrho' n}  \, ds \Big)^{\frac{1}{\varrho'}}
\Big(\int_{\R^d} (1+2^j |s|)^{(-n r' + \delta)b} \, ds \Big)^{\frac{1}{r'b}}
\\ & \leq C 2^{-jd(\frac{1}{r'}- \frac{1}{\varrho'})}  \Big(\int_{\R^d} \|K_j(s)x\|^{\varrho'}  (1+2^j |s|)^{\varrho' n}  \, ds \Big)^{\frac{1}{\varrho'}},
\end{align*}
since $(-nr' + \delta)b<-d$, or equivalently $n>  \frac{d}{\varrho}  - \frac{d}{r} + \frac{\delta}{r'}$, for $\delta>0$ small enough.
Writing $(1+2^j |s|)^{n}\leq C \sum_{|\gamma|\leq n} |(2^j s)^{\gamma}|$ and using the Fourier type $\varrho$ of $Y$, it follows that
\begin{align*}
T_{1,j}& \leq C 2^{-j(\frac{1}{\varrho}- \frac{1}{r})}  \sum_{|\gamma|\leq n} 2^{j|\gamma|} \Big(\int_{\R^d} \|s^{\gamma} K_j(s)x\|^{\varrho'}  \, ds \Big)^{\frac{1}{\varrho'}}
\\ & \leq C 2^{-j( \frac{1}{\varrho}- \frac{1}{r})}  \sum_{|\gamma|\leq n} 2^{j|\gamma|}
\Big(\int_{\R^d} \|\partial^{\gamma} m_j(\xi)x\|^{\varrho}  \, d\xi \Big)^{\frac{1}{\varrho}}.
\end{align*}
Using the Leibniz rule, the support condition of $\wh{\zeta_{j}}$ and the assumption (M1)$_{r,\varrho,n}$, as in \cite[Theorem 5.2.7]{Grafakos08} we find that
\begin{align*}
\Big(\int_{\R^d} \|\partial^{\gamma} m_j(\xi)x\|^{\varrho}  \, d\xi \Big)^{\frac{1}{\varrho}} \leq C M_1 2^{jd(\frac{1}{\varrho} - \frac{1}{r})}  2^{-j|\gamma|}\|x\|.
\end{align*}
Therefore, \eqref{eq:HormKern1} follows if we combine the estimates. The proof of \eqref{eq:HormKern2} is similar. The extra factor $2^{-j}$ cancels out because of the extra factor $|\xi|$ which comes from the Fourier transform of $\nabla K_j$.

It remains to check that $K$ satisfies (H)$_{r'}$. By the triangle inequality it suffices to prove that
\begin{align}\label{eq:sumjZa}
\sum_{j\in \Z} \Big(\int_{|s|\geq 2|t|} \|K_j(s-t)x - K_j(s)x\|^{r'} \, ds\Big)^{\frac{1}{r'}}\leq CM_1\|x\|_{X}\quad (t\neq 0, x\in X)
\end{align}
for a constant $C\geq 0$ independent of $m$.
Let $x\in X$ and $t\in \R^d\setminus\{0\}$, and choose $k\in \Z$ such that $2^{-k}\leq |t|\leq 2^{-k+1}$. Then, by  \eqref{eq:HormKern1}, for the part of the sum with $j>k$ we find
\begin{align*}
\sum_{j>k} \Big(\int_{|s|\geq 2|t|} &\|K_j(s-t)x - K_j(s)x\|^{r'} \, ds\Big)^{\frac{1}{r'}} \\ & \leq \sum_{j>k} 2\Big(\int_{|s|\geq |t|} \|K_j(s)x\|^{r'} \, ds\Big)^{\frac{1}{r'}}
\\ & \leq \sum_{j>k} 2\Big(\int_{|s|\geq |t|} \|K_j(s)x\|^{r'} \frac{(1+2^j |s|)^{\delta}}{(1+2^j |s|)^{\delta}} \, ds\Big)^{\frac{1}{r'}}
\\ & \leq \sum_{j>k}  \frac{2 C_d M_1 \|x\|}{(1+2^j |t|)^{\delta}}\leq  \sum_{j>k}  \frac{2 C_d M_1 \|x\|}{(1+2^{j-k})^{\delta}} = C M_1 \|x\|
\end{align*}
For the part with $j\leq k$, it follows from Minkowski's inequality and \eqref{eq:HormKern2} that
\begin{align*}
\sum_{j\leq k} \Big(\int_{|s|\geq 2|t|} &\|K_j(s-t)x - K_j(s)x\|^{r'} \, ds\Big)^{\frac{1}{r'}}
\\ & = \sum_{j\leq k} \Big(\int_{|s|\geq 2|t|} \Big\|\int_0^1 - t \cdot \nabla K_j(s-\theta t)x \, d\theta\Big\|^{r'} \, ds\Big)^{\frac{1}{r'}}
\\ & \leq \sum_{j\leq k} |t|  \int_0^1  \Big(\int_{\R^d}  \|\nabla K_j(s-\theta t)x\|^{r'} \, ds\Big)^{\frac{1}{r'}} \, d\theta
\\ & \leq \sum_{j\leq k} 2^{-k+1} \int_0^1  \Big(\int_{\R^d} \|\nabla K_j(s-\theta t)x\|^{r'} (1+2^j |s-\theta t|)^{\delta}   \, ds\Big)^{\frac{1}{r'}}\, d\theta
\\ & \leq \sum_{j\leq k} 2^{-k+1} 2^{j} C_d M_1\|x\| = CM_1\|x\|.
\end{align*}
Now Corollary \ref{cor:Hormander} concludes the proof.

(2):
Again, it suffices to prove \eqref{eq:toprovebddextr2} with $m$ replaced by $m^{(N)}$ with constants independent of $N$. So fix $N\in \N$ and write $m = m^{(N)}$ and $K = K^{(N)}$ again. It follows as in the proof of (1) that $K$ satisfies the hypotheses of Corollary \ref{cor:HormanderII}.
\end{proof}

\subsection{Applications\label{subsec:ExtraAppl}}

A straightforward application of Corollary \ref{cor:extrapol m uniform} is that under suitable smoothness conditions on $m$ one can extrapolate the result of Theorem \ref{thm:Lp-Lq multipliers type} (assuming $X$ has type $p_0$ and $Y$ has cotype $q_0$) to all values of $1<p\leq q<\infty$ with $\frac{1}{p} - \frac1q = \frac{1}{p_0}-\frac{1}{q_0}$.

Next we wish to extrapolate the result of Theorem \ref{main result homogeneous Besov multipliers type}. To this end, we first extrapolate Proposition \ref{compact Fourier support multipliers}, as it stands at the basis of our results on Besov spaces.

\begin{lemma}\label{lem:alt compact Fourier support multipliers}
Let $X$ be a Banach space with type $p_0\in[1,2]$ and $Y$ a Banach space with cotype $q_0\in[2,\infty]$, and let $r\in[1,\infty]$ be such that $\tfrac{1}{p_0}-\tfrac{1}{q_0}=\tfrac{1}{r}$. Let $m:\Rd\setminus\{0\}\to\La(X,Y)$ be an $X$-strongly measurable map such that $\{m(\xi)\!\mid\! \abs{\xi}\in [2^{k-2},2^{k+2}]\}\subseteq\calL(X,Y)$ is $\gamma$-bounded by $2^{-k\sigma}M$ for some $k\in\Z$, $\sigma\in\R$ and $M\geq0$.
Suppose that $X$ and $Y$ both have Fourier type $\varrho\in [1, 2]$ with $\varrho\leq r$, and let $n:=\lfloor d(\tfrac{1}{\varrho}-\tfrac{1}{r})\rfloor+1$.
Assume that, for a constant $C\geq 0$ and for all multi-indices $|\alpha|\leq n$,
\begin{equation}\label{eq:mihlin3}
\|\partial^{\alpha} m(\xi)\|_{\La(X,Y)}\leq C|\xi|^{-|\alpha| - \sigma}\qquad (\abs{\xi}\in [2^{k-2},2^{k+2}]).
\end{equation}
Then, for all $1<p\leq q<\infty$ such that $\frac1p-\frac1q=\tfrac{1}{r}$, the operator $T_m:\Sw_{J_k}(\Rd;X)\to\Sw'(\Rd;Y)$ extends to a bounded operator $\widetilde{T_{m}}\in\La(\Ell^{p}_{J_k}(\R^d;X),\Ell^{q}_{J_k}(\R^d;Y))$. Moreover,
\begin{align}\label{eq:boundedcompactsupp}
\|\widetilde{T_m}(f)\|_{\Ellq(\R^d;Y)}\leq C_{d,p,q} (C + M) 2^{kd/r} 2^{-k \sigma}  \,\|f\|_{\Ellp(\R^d;X)}
\end{align}
for all $f\in\Ell^{p}_{J_k}(\R^d;X)$ and some $C_{d,p,q}\geq 0$.
\end{lemma}

\begin{proof}
Fix $\zeta\in \Sw(\R^d)$ such that $\zeta(\xi) =1$ if $\abs{\xi}\in [\tfrac{1}{2},2]$ and $\zeta(\xi) = 0$ if $\abs{\xi}\notin [\tfrac{1}{4},4]$. Let $\zeta_k(\xi) := \zeta(2^{-k} \xi)$.
If we set $m_k := \zeta_k m$ then clearly $T_{m_k} f = T_m f$ for $f\in \Sw_{J_{k}}(\R^d;X)$.
Using Leibniz's rule one may check that $m_k$ still satisfies \eqref{eq:mihlin3} with a bound independent of $k$. Note that $\check{\zeta_{k}}\ast f\in\Sw_{[-2^{k+2},2^{k+2}]^{d}}(\Rd;X)$ for each $f\in\Sw(\Rd;X)$. By Proposition \ref{compact Fourier support multipliers},
\begin{align*}
\|T_{m_k}(f)\|_{\Ell^{q_0}(\R^d;Y)} & = \|T_{m}(\check{\zeta_k}\ast f)\|_{\Ell^{q_0}(\R^d;Y)}
\\ & \leq \tau_{p_0,X}c_{q_0,Y}2^{(k+3)d/r}2^{-k\sigma}M\|f\|_{\Ell^{p_0}(\R^d;X)},
\end{align*}
hence $T_{m_k}$ extends to a bounded linear operator from $\Ell^{p_0}(\R^d;X)$ into $\Ell^{q_0}(\R^d;Y)$. By Corollary \ref{cor:extrapol m uniform} \eqref{eq:boundedcompactsupp} holds with $T_{m}$ replaced by $T_{m_k}$ (even without the support condition on $f$). Specializing to $f\in \Sw_{J_{k}}(\R^d;X)$ and using that $T_{m_k} f = T_m f$, the required result follows from Lemma \ref{lem:approxsmooth}.
\end{proof}

Now we can extrapolate Theorem \ref{main result homogeneous Besov multipliers type} to other values of $p$ and $q$:

\begin{theorem}\label{main result homogeneous Besov multipliers typeextra}
Assume the conditions of Theorem \ref{main result homogeneous Besov multipliers type} and suppose that $X$ and $Y$ both have Fourier type $\varrho\in [1, 2]$ with $\varrho\leq r$. Let $n:=\lfloor d(\tfrac{1}{\varrho}-\tfrac{1}{r})\rfloor+1$. Assume that, for a constant $C\geq 0$ and for all $k\in\Z$ and all multi-indices $|\alpha|\leq n$,
\begin{equation}\label{eq:mihlin4}
\|\partial^{\alpha} m(\xi)\|_{\La(X,Y)}\leq C|\xi|^{-|\alpha| - \sigma}\qquad (\abs{\xi}\in [2^{k-2},2^{k+2}]).
\end{equation}
Then, for all $1<\tilde{p}\leq \tilde{q}<\infty$ satisfying $\frac{1}{\tilde{p}}-\frac{1}{\tilde{q}}=\frac{1}{r}$, the operator $T_m$ is bounded from $\dot{\Be}^{s}_{\tilde{p},v}(\Rd;X)$ to $\dot{\Be}^{s+\sigma-d/r}_{\tilde{q},w}(\Rd;Y)$.
\end{theorem}

\begin{proof}
This is immediate from Lemma \ref{lem:alt compact Fourier support multipliers} and Proposition \ref{prop:multipliers_on_homogeneous_Besov_spaces}.
\end{proof}

\begin{remark}
Lemma \ref{lem:alt compact Fourier support multipliers} also holds with $J_k$ (see \eqref{eq:dyadicJ}) replaced by $I_k$ (see \eqref{dyadic annuli}) if instead of \eqref{eq:mihlin3} one assumes for $k=0$ that $\{m(\xi)\mid \abs{\xi}\in[0,4]\}\subseteq\La(X,Y)$ is $\gamma$-bounded and that, for all $|\alpha|\leq n$,
\[
\|\partial^{\alpha} m(\xi)\|_{\La(X,Y)}\leq C(1+|\xi|)^{-|\alpha| - \frac{d}{r}}\qquad (\xi\in \R^d) \qquad (\abs{\xi}\in [0,4]).
\]
Hence under this additional assumption Theorem \ref{main result Besov multipliers type} also extrapolates to all $p$ and $q$ as above.
\end{remark}

\noindent
{\bf Acknowledgment} The authors would like to thank the referee for his/her comments.

\end{document}